\documentclass{bcp}
\usepackage{amsmath,amsthm}
\usepackage{amssymb}
\usepackage{xcolor}
 
\usepackage{graphicx}

\newtheorem{theorem}{Theorem}[section]
\newtheorem{proposition}[theorem]{Proposition}
\newtheorem{lemma}[theorem]{Lemma}
\newtheorem{corollary}[theorem]{Corollary}

\theoremstyle{definition}
\newtheorem{definition}[theorem]{Definition}
\newtheorem{example}[theorem]{Example}
\newtheorem{remark}[theorem]{Remark}

\newtheorem{assumption}[theorem]{Assumption} 

\begin{document}

\mathclass{Primary 60G40; Secondary 91A25.}

\abbrevauthors{S. Christensen and K. Kristoffer Lindensj\"o}
\abbrevtitle{Time-inconsistent stopping}

\title{Time-inconsistent stopping, myopic adjustment \& equilibrium stability:  with a mean-variance application}

\author{S\"oren Christensen}
\address{Department of Mathematics, Kiel University\\
Ludewig-Meyn-Str. 4, D-24098 Kiel, Germany\\
E-mail: christensen@math.uni-kiel.de}

\author{Kristoffer Lindensj\"o}
\address{Department of Mathematics, Stockholm University\\
SE-106 91 Stockholm, Sweden\\
E-mail: kristoffer.lindensjo@math.su.se}

\maketitlebcp

\begin{abstract}
For a discrete time Markov chain and in line with Strotz' consistent planning we develop a framework for problems of optimal stopping that are time-inconsistent due to the consideration of a non-linear function of an expected reward. We consider pure and mixed stopping strategies and a (subgame perfect Nash) equilibrium. 
We provide different necessary and sufficient equilibrium conditions including a verification theorem.  
Using a fixed point argument we provide equilibrium existence results. We adapt and study the notion of the myopic adjustment process and introduce different kinds of equilibrium stability. 
We show that neither existence nor uniqueness of equilibria should generally be expected. 
The developed theory is applied to a mean-variance problem and a variance problem.  
\end{abstract}


\section{Introduction} \label{sec:intro}
Consider a stochastic process $X$ on a state space $E$ and the problem of finding a stopping time $\tau$ that maximizes
\begin{align}
\begin{split}
J_{\tau}(x)& := \mathbb{E}_x(f(X_\tau)) + g\left(\mathbb{E}_x(h(X_\tau))\right),  \enskip X_0=x\\
& \mbox{ where } f,h:E\rightarrow \mathbb{R} \mbox{ and } g:\mathbb{R}\rightarrow \mathbb{R}. \label{payoff}
\end{split}
\end{align}
This problem is in general time-inconsistent in the sense that if a stopping rule is optimal for a particular initial value $x$ then it is generally not optimal for $x' \neq x$; the reason being that $g$ may be non-linear. Note that we may formulate both mean-variance and variance stopping problems as special cases of \eqref{payoff}, see Section \ref{examples}.  

The consistent planning approach to  time-inconsistent problems pioneered by Strotz and Selten \cite{selten1965spieltheoretische,selten1975reexamination,strotz} corresponds --- in a stopping problem context --- to viewing \eqref{payoff} from the perspective of a person who decides when to stop $X$ but whose preferences, due to the time-inconsistency, change as $X$ evolves; and therefore \eqref{payoff} is viewed as an intrapersonal non-cooperative stopping game. The approach is formalized by formulating an appropriate mathematical definition of a subgame perfect Nash equilibrium. We refer to e.g. \cite{tomas-disc,christensen2017finding,christensen2018time,lindensjo2017timeinconHJB} for more comprehensive interpretations of the equilibrium approach to time-inconsistent problems.

The present paper is structured as follows. 
In Section \ref{rel-lit} we motivate the study of time-inconsistent stopping by formulating three types of problems that are studied in finance and economics and review some of the related literature. 
In Section \ref{sec:problem} we define mixed and pure stopping strategies and the equilibrium. In Section \ref{time-consistent case} we show that the definition of equilibrium coincides with standard optimality when the problem is time-consistent (i.e. when $g=0$ in \eqref{payoff}).  
In Section \ref{results-sec} we derive several results with necessary and sufficient equilibrium conditions including a verification theorem. 
In Section \ref{fixed-point-sec} we provide a fixed point problem characterization of equilibrium and related equilibrium existence results.  
In Section \ref{sec:myopic} we adapt and study the notion of a myopic adjustment process. We also define and study different notions of equilibrium stability.
In Sections \ref{mean-var-prob} and \ref{var-prob} the developed theory is applied to mean-variance and variance optimization.  
In Section \ref{uniqueness-existence-examples} we show that an equilibrium does not necessarily exist and if it does then it is not necessarily unique.  
In Section \ref{equilibrium-discussion} we discuss the framework of the present paper in relation to the literature.  
The appendix contains some technical results.

\subsection{Time-inconsistency in economics \& related literature} \label{rel-lit}
In order to motivate the study of time-inconsistent stopping problems in general we here present three simple examples which correspond to time-inconsistent problems commonly studied in finance and economics. Similar presentations are contained in 
\cite{christensen2017finding,christensen2018time} while \cite{tomas-disc,lindensjo2017timeinconHJB} present these problems in a regular stochastic control framework. Note that of the three kinds of problems described in this section only mean-variance optimization can directly be studied within the framework of the present paper.
In \cite{christensen2018time} we develop a general framework for the equilibrium approach to time-inconsistent stopping problems of the type in the present paper for a one-dimensional diffusion. 
We remark that time-inconsistent problems can also be studied using the pre-commitment approach and the dynamic optimality approach. 
In the context of the present paper the pre-commitment approach corresponds to maximizing \eqref{payoff} for a particular $x$. The dynamic optimality approach was invented in \cite{pedersen2016optimal,pedersen2013optimal} and corresponds to choosing a strategy that is optimal with respect to all present states.

\textbf{Mean-variance optimization}:  
In a stopping problem context the mean-variance problem can be motivated with the following example. Suppose an investor wants to sell an asset whose price follows a stochastic process $X$. Suppose the investor wants, for any particular $x$, to use a selling strategy, i.e. a stopping time $\tau$, that maximizes
\begin{align} 
\mathbb{E}_x(X_\tau) - \gamma \mbox{Var}_x(X_\tau), \notag
\end{align}
for a fixed parameter $\gamma>0$ corresponding to risk aversion. The interpretation is that the investor wants a large expected payoff but is averse to risk measured in terms of selling price variance. A mean-variance stopping problem is studied in Section \ref{mean-var-prob}. In \cite[Section 4.2]{christensen2018time} a mean-variance stopping problem for a geometric Brownian motion is studied using the equilibrium approach. In \cite{bayraktar2018} a mean-standard deviation and mean-variance stopping problem for a discrete time Markov chain is studied using the equilibrium approach;{
we remark that a main part of \cite{bayraktar2018} considers \emph{liquidation strategies}, whose interpretation, in the context of an asset selling problem, is that the investor may sell the asset over several time periods.} 
 In \cite{pedersen2016optimal} a mean-variance stopping problem for a geometric Brownian motion is studied using a precommitment approach and the dynamic optimality approach. 
We note that there is a large literature on mean-variance optimization especially for regular stochastic control (often corresponding to dynamic asset portfolio selection), see e.g.  
\cite{bensoussan2014time,
bielecki2005continuous,
tomas-mean,
christensen2018time,Czichowsky,
delong2015instantaneous,
he2013optimal,
kronborg2015inconsistent,
landriault2018equilibrium,
li2013optimal,
pedersen2013optimal,schoneborn2015optimal,van2018time,
vigna2014efficiency,
yan2019open,
zeng2016robust}.

\textbf{Endogenous habit formation}: An example of this kind of problem is a version of the asset selling problem introduced above corresponding to
\begin{align} 
\mathbb{E}_x(F(X_\tau,x)),\notag
\end{align}
where $F(\cdot,x)$ is a utility function parametrized by $x$. The interpretation is that the current price of the asset determines the utility function of the investor. In \cite{christensen2017finding} we develop a general framework for the equilibrium approach to time-inconsistent stopping problems of the endogenous habit formation type for a continous time Markov process and in 
\cite[Example 5.8.]{christensen2017finding} an endogenous habit formation asset selling problem is studied. Endogenous habit formation is also studied in e.g. 
\cite{tomas-continFORTH,detemple1992optimal,englezos2009utility,yu2017optimal}.

\textbf{Non-exponential discounting}: An example of this kind of problem is the version of the asset selling problem corresponding to
\begin{align} 
\mathbb{E}_{t,x}(\delta(\tau-t)F(X_\tau)),\notag
\end{align}
where $\delta(\cdot)$ is a non-exponential discounting function (i.e a non-increasing function taking values in $[0,1]$ with $\delta(0)=1$). Non-exponential discounting stopping problems are studied in 
\cite{bayraktar2019notions,
huang2018time,
huang2017optimal,
huang2017optimalDISC}. They can also be studied within the continuous time framework of \cite{christensen2017finding}. Non-exponential discounting is also studied in e.g. \cite{alia2019non,balbus2018markov}.

Now we mention some other related problems studied in the recent literature. A time-inconsistent stopping problem under model ambiguity is studied using the equilibrium approach in \cite{huang2019optimal}. Conditional optimal stopping is studied using the equilibrium approach in \cite{nutz2019onditional}. 
Precommitment and naive strategies for optimal  exit times for gambling are studied in \cite{he2019optimal}. 
Optimal stopping under probability distortion is studied with a precommitment approach in \cite{xu2013optimal}. In \cite{huang2019general} a general framework for naive and equilibrium strategies for time-inconsistent stopping problems for a diffusion is developed and applied to probability distortion. The principle of smooth pasting for a particular problem is considered in \cite{tan2018failure}. \cite{Duraj2017Optimal} studies a framework under a general preference structure.
A version of the classical dividend problem with a time-inconsistent restriction is studied in \cite{christensen2019dividend} using both the precommitment and consistent planning approach.

Further references to the literature are found throughout the paper. A recent survey of time-inconsistent stochastic control is \cite{yan2019time}.

\section{Problem formulation} \label{sec:problem}
We consider the time-inconsistent stopping problem \eqref{payoff} for a discrete time strong time-homogeneous Markov chain $X=\{X_n\}$, $n\in \mathbb{N}_0$, taking values in a finite state space $E$ with $N$ elements. We also consider a stochastic process $\{Y_n\}$, $n\in \mathbb{N}_0$, where each $Y_n$ is uniformly distributed on $[0,1]$ and independent of $X$ and of every $Y_k$, $k\neq n$. We denote by $\mathbb{P}_x$ the measure under which $X_0=x\in E$ a.s. The associated expectations are denoted by $\mathbb{E}_x$. 

As a notational convenience we consider an ordering of the state space $E$ and identify $N$-dimensional vectors, for instance 
$\textbf{p}$, with functions on the state space, i.e.

\[\textbf{p} = \left(p_{x_1},p_{x_2},...,p_{x_{N-1}},p_{x_N}\right)^T 
= \left(p_x\right)_{x\in E}.\]

\begin{definition}[Mixed stopping strategies] \label{def:mix-strat} 
A vector $\textbf{p}\in [0,1]^N$ is said to be a mixed (Markov) stopping strategy and 
\[\tau_\textbf{p}:= \min\{n\geq 0: Y_{n} \leq p_{X_n}\}\]
is said to be a mixed (Markov) strategy (profile) stopping time.
\end{definition}
Note that $\tau_\textbf{p}$ is a stopping time with respect to the filtration $\sigma(X_0,...,X_n,Y_0,...,Y_n)$, $n\in \mathbb{N}_0$.

\begin{definition}[Pure stopping strategies] 
A mixed stopping strategy $\textbf{p}$ is said to be a pure stopping strategy if $\textbf{p} \in \{0,1\}^N$.
\end{definition}

\begin{remark} [Interpretation] 
For a mixed stopping strategy ${{\textbf{p}}}$ and any $n$ the conditional probability of stopping $X$ at $n$, \emph{before} having observed $Y_n$, given that $X$ has not been stopped before $n$, is $p_{X_n}$. In this sense a stopping strategy ${{\textbf{p}}}$ corresponds to using the random variable $Y_n$ as a randomization device for the stopping decision made at $n$.   
Note that the randomization device can be interpreted as flipping a biased coin at each $n$ and stopping at $n$ if the outcome is, say, heads, where the probability of heads is $p_y$ if the observed state is $X_n=y$. 
For a pure stopping strategy the conditional probability of stopping $X$ at $n$ given that $X$ has not been stopped before $n$ is either one or zero; and in this sense the decision to stop or not at $n$ depends only on the payoff relevant quantity $X_n$ without randomization. 
For a more thorough description of the game theory terms used in this section in another time-inconsistent stopping context see \cite{christensen2017finding}. 
\end{remark} 
\begin{remark} Note that $\textbf{p}$ is a complete specification of the strategies of all players in the game and that $\textbf{p}$ is in this sense a strategy profile, although we refer to $\textbf{p}$ as a stopping strategy{ to be more in line with the existing literature}. 
\end{remark} 

Since the distribution of $\tau_\textbf{p}$ is determined by $\textbf{p}$ we typically perform the analysis of the present paper from the viewpoint of stopping strategies $\textbf{p}$. 
Hence, instead of $J_{\tau_\textbf{p}}(x)$ we write $J_\textbf{p}(x)$, cf. \eqref{payoff}.

\begin{definition}[Equilibrium]\label{def:equ_stop_time} A stopping strategy ${\hat {\textbf{p}}}$ is said to be a (subgame perfect Nash) equilibrium if 
\begin{align}
\label{Eq} \tag{EqI} 
\begin{split}
J_{\hat {\textbf{p}}}(x) &\geq qf(x)  + (1-q)\mathbb{E}_x\left(\mathbb{E}_{X_1}(f(X_{\tau_{\hat {\textbf{p}}}})) \right)\\
& + g\left(qh(x) + (1-q)\mathbb{E}_x\left(\mathbb{E}_{X_1}(h(X_{\tau_{\hat {\textbf{p}}}})) \right)\right), 
\mbox{ for all $q \in [0,1]$ and all $x \in E$.} 
\end{split}
\end{align}
The equilibrium is said to be pure if $\hat {\textbf{p}}$ is pure. 
If $\hat {\textbf{p}}$ is an equilibrium then 
$\tau_{\hat {\textbf{p}}}$ is said to be the (corresponding) equilibrium stopping time and 
$J_{\hat {\textbf{p}}}$ 
 is said to be the (corresponding) equilibrium value function.
\end{definition}

\begin{remark} [Interpretation]
The interpretation of the expression in the right hand side of \eqref{Eq} --- or equivalently of $K(x,q,{\hat{\textbf{p}}})$, see 
\eqref{K-func} and \eqref{Eq'}--\eqref{Eq'''} below --- is that it is the value obtained at $x$ when stopping at $x$ with probability $q$ given that the strategy ${\hat{\textbf{p}}}$ is used subsequently. The interpretation of the expression in the left hand side of 
\eqref{Eq}-- \eqref{Eq''} 
 is that it is the value obtained at $x$ when using the strategy  ${\hat{\textbf{p}}}$ given that ${\hat{\textbf{p}}}$  is used subsequently. 
The interpretation of an equilibrium $\hat {\textbf{p}}$ is therefore that, for each $x$, there is the \emph{possibility} to deviate from $\hat {\textbf{p}}$ at $x$ in the sense of using any other biased coin --- cf. $q$ in \eqref{Eq} --- to determine whether to stop $X$ at the present time or not; but that such a deviation is never preferred to using ${\hat {p}}_x$ given that 
${\hat {\textbf{p}}}$ is used at all subsequent dates.
\end{remark}

This paper is devoted to the question of how to find equilibria as defined above. Throughout  the paper we suppose the following assumptions hold: 
\begin{assumption} \label{g-assum}
The function $g:\mathbb{R}\rightarrow \mathbb{R}$ in \eqref{payoff} is continous.
\end{assumption}
\begin{assumption} \label{absorb-assum}
$X$ is an absorbing Markov chain; that is, for each $X_0=x\in E$, there is (at least) one absorbing state in $E$ that $X$ reaches with positive probability in a finite number of steps. 
\end{assumption}
We use the convention
\begin{align}  \notag
\mbox{$f(X_\tau):= \lim_{n\rightarrow\infty}f(X_n)$ and $h(X_\tau):= \lim_{n\rightarrow\infty}h(X_n)$ on $\{\tau = \infty\}$,}
\end{align}
where the limits exist due to Assumption \ref{absorb-assum}, and the notation
\begin{align}  \label{phi-psi-notation}
\phi_\textbf{p}(x):= 
\mathbb{E}_x(f(X_{\tau_\textbf{p}})) \mbox{ and } \psi_\textbf{p}(x):= \mathbb{E}_x(h(X_{\tau_\textbf{p}})),
\end{align}
which we note implies that
\begin{align} \notag 
J_\textbf{p}(x) =  \phi_\textbf{p}(x) + g\left(\psi_\textbf{p}(x)\right).
\end{align}
	We remark that if $p_x=0$ for each $x\in E$ then $\tau_\textbf{p}=\infty$ a.s. meaning that $X$ is never stopped. However, in this case $X$ eventually reaches an absorbing state by Assumption \ref{absorb-assum} and therefore stops in this sense. We remark the related fact that $\phi_\textbf{p}$ and $\psi_\textbf{p}$ are independent of the value of $p_x$ whenever $x$ is an absorbing state. 
%

We also use the notation
\begin{align}
\label{K-func}
\begin{split}
&K(x,q,{\textbf{p}})\\
&\enskip :=
qf(x)  + (1-q)\mathbb{E}_x\left(\phi_{\textbf{p}}(X_1)\right) + g\left(qh(x) + (1-q)\mathbb{E}_x\left( \psi_{\textbf{p}} (X_1) \right)\right).
\end{split}
\end{align} 
We now provide three equivalent equilibrium definitions that will be used in the sequel. 
\begin{proposition} \label{equivalent-eq-def} 
Each one of the following conditions is equivalent to the equilibrium condition \eqref{Eq}:
\begin{align}
\phi_{\hat {\textbf{p}}}(x)  + g\left(\psi_{\hat {\textbf{p}}}(x)\right)  &\geq K(x,q,{\hat {\textbf{p}}}), \mbox{ for all $q \in [0,1]$ and all $x \in E$.} \label{Eq'} \tag{EqII}\\
\phi_{\hat {\textbf{p}}}(x)  + g\left(\psi_{\hat {\textbf{p}}}(x)\right)&= \max_{q\in [0,1]}K(x,q,{\hat {\textbf{p}}}), \mbox{ for all $x \in E$.}\label{Eq''} \tag{EqIII}\\
\hat {p}_x & \in \mbox{\normalfont arg}\hspace{-1mm}\max_{q \in [0,1]} K(x,q,{\hat {\textbf{p}}}), \mbox{ for all $x \in E$.} 	
\label{Eq'''} \tag{EqIV} 
\end{align}
\end{proposition}
\begin{proof}
Use the notation \eqref{phi-psi-notation} and \eqref{K-func} to see that the first result holds. The second and third results follow from the first result and the observation that if we set $q=\hat {p}_x$ in \eqref{Eq'} then equality is attained, cf. Lemma \ref{phi-p-expressed-with-Exp}. 
\end{proof}
 
\begin{remark}{
It is possible to slightly relax Assumptions \ref{g-assum} and \ref{absorb-assum} at the cost of increasing the amount of technical details. In particular, the discussion of the case $\tau_{\mathbf p}=\infty$ is getting more difficult without absorbing states. Here, a careful definition of limits of the form $\lim_{n\rightarrow\infty}f(X_n)$ is necessary. We have chosen not to include this discussion in order to focus on the main ideas and not overburden the presentation. \\
However, introducing a discount factor solves this issue. This is directly possible in the framework introduced above. Indeed, if we start with a general Markov chain $X$ on\ $E$ with possibly no absorbing states, we consider the associated geometrically killed Markov chain $\tilde X$ on $E\cup\{\Delta\}$ with killing rate $q\in(0,1)$, where $\Delta$ is an isolated point such that all rewards are 0 in $\Delta$. Then, $\tilde X$ fulfills Assumption \ref{absorb-assum} and, when we assume that $g(0)=0$ w.l.o.g.,
\begin{align} \notag
\begin{split}
\tilde J_{\tau}(x)&:= \mathbb{E}_x(f(\tilde X_\tau)) + g\left(\mathbb{E}_x(h(\tilde X_\tau))\right)\\
& =\mathbb{E}_x(q^\tau f(X_\tau)) + g\left(\mathbb{E}_x(q^\tau h(X_\tau))\right).
\end{split}
\end{align}
The choice to consider a finite state space has been made in order to not overburden the paper with technical details; in particular, we expect it to be possible to consider a countably infinite state space, at the cost of increasing the amount of technical details, regarding e.g. how Assumption \ref{absorb-assum} should in this case be formulated.}
\end{remark}

\subsection{The time-consistent case} \label{time-consistent case}
If we consider a standard stopping problem, corresponding to maximization for \eqref{payoff} with $g=0$, then an equilibrium stopping time has the desirable property of being characterized as an --- in the usual sense --- \emph{optimal} stopping time:

\begin{theorem} \label{time-con-thm} Suppose $g=0$. Then, $\tau_{\hat {\textbf{p}}}$ is an equilibrium stopping time for \eqref{payoff} if and only if $\tau_{\hat {\textbf{p}}}$ is an optimal stopping time for \eqref{payoff}.
\end{theorem}
\begin{proof}
$g=0$ implies that the equilibrium condition \eqref{Eq} can be written as
\begin{align} \notag
J_{\hat {\textbf{p}}}(x)  \geq qf(x) + (1-q)\mathbb{E}_x\left(  J_{\hat {\textbf{p}}}(X_1) \right), \mbox{ for all $q \in [0,1]$ and all $x \in E$,}
\end{align}
or equivalently as
\begin{align} \notag
J_{\hat {\textbf{p}}}(x)  \geq \max\{f(x),\mathbb{E}_x\left(  J_{\hat {\textbf{p}}}(X_1) \right)\}, \mbox{ for all $x \in E$.}
\end{align}
We find that $J_{\hat {\textbf{p}}}(x)$ is an equilibrium value function if and only if
(i) $J_{\hat {\textbf{p}}}(x)$ is excessive for $X$, 
(ii)  $J_{\hat {\textbf{p}}}(x)$ majorizes $f(x)$, and
(iii)  $J_{\hat {\textbf{p}}}(x)=\mathbb{E}_x(f(X_{\tau_{\hat {\textbf{p}}}}))$; i.e. if and only if $J_{\hat {\textbf{p}}}(x)$ is the optimal value function of the problem \eqref{payoff} with $g=0$ (by well-known results from the general theory of optimal stopping, see e.g. \cite{MR2374974}).
\end{proof}

\section{Necessary and sufficient equilibrium conditions} \label{results-sec}

It is instructive to note that the right hand side of the equality in \eqref{Eq''}, or equivalently \eqref{Eq'''}, is for each fixed $x$ and $\hat {\textbf{p}}$ an elementary optimization problem of a function $[0,1]\rightarrow\mathbb{R}$, cf. \eqref{K-func}; which in particular can be written as 
\begin{align} \notag
\max_{q\in [0,1]}  \{qc_1  + (1-q)c_2 + g\left(qc_3 + (1-q)c_4\right)\},
\end{align}
where $c_1,...,c_4$ are constants (depending on $x$ and $\hat {\textbf{p}}$). Using this observation we immediately obtain:

\begin{theorem} \label{lem:nec-cond}  \enskip
\begin{itemize}
\item Suppose $g$ in    \eqref{payoff} is differentiable. Then, a necessary  condition for a stopping strategy ${\hat {\textbf{p}}}$ to be an equilibrium is that, for each $x \in E$, the following inequalities hold  and (at least) one of them holds with equality:

\begin{align}	
\phi_{\hat {\textbf{p}}}(x) + g\left(\psi_{\hat {\textbf{p}}}(x)\right) & \geq f(x) + g(h(x))  \label{I} 
\\
\phi_{\hat {\textbf{p}}}(x) + g\left(\psi_{\hat {\textbf{p}}}(x)\right)& 
\geq \mathbb{E}_x\left(   \phi_{\hat {\textbf{p}}}(X_1)  \right) + g\left(\mathbb{E}_x\left( \psi_{\hat {\textbf{p}}}(X_1) \right)\right) \label{II}  
\\
 \label{III}  
\begin{split}
| f(x)  - \mathbb{E}_x\left(  \phi_{\hat {\textbf{p}}}(X_1)  \right) &+ g'\left(\hat {p}_xh(x) + 
(1-\hat {p}_x)\mathbb{E}_x\left( \psi_{\hat {\textbf{p}}}(X_1)  \right)\right)\\ 
&\times \left( h(x)  - \mathbb{E}_x\left( \psi_{\hat {\textbf{p}}}(X_1)  \right) \right)| \geq  0.
\end{split}
\end{align}
\item  Suppose $g$  in \eqref{payoff} is twice differentiable. Then, a necessary  condition for a stopping strategy ${\hat {\textbf{p}}}$ to be an equilibrium is that it 
for each $x\in E$ with $\hat {p}_x \in (0,1)$ (if such points exist) holds that,
\begin{align}	\notag
g''\left(\hat {p}_xh(x) + (1-\hat {p}_x)\mathbb{E}_x\left( \psi_{\hat {\textbf{p}}}(X_1)  \right)\right)\leq 0.
\end{align}	 
\end{itemize}
\end{theorem}
\begin{proof} Inequalities \eqref{I} and \eqref{II} are obtained by setting $q=1$ and $q=0$ in \eqref{Eq'}, respectively. Inequality \eqref{III} is trivial. If none of \eqref{I}--\eqref{III} holds with equality then  $\hat {p} _x$ cannot satisfy \eqref{Eq'''}; to see this use e.g. that \eqref{III} is essentially a first order condition for the maximization in \eqref{Eq'''}. Hence, the first result holds. The second result is proved similarly; it is essentially a second order condition for the maximization in \eqref{Eq'''}. 
\end{proof}

We now provide an equilibrium verification theorem.

\begin{definition} \label{characterizing-equation-def} 
Two functions $\phi:E\rightarrow \mathbb{R}$ and $\psi:E\rightarrow \mathbb{R}$ are said to be a solution to the \emph{characterizing equation} if, for all $x\in E$, 
\begin{align}	
 &\phi(x) + g\left(\psi(x)\right)\notag\\
 &\quad = \max_{q\in [0,1]}\left\{qf(x)  + (1-q)\mathbb{E}_x\left(\phi(X_1)\right) + g\left(qh(x) + (1-q)\mathbb{E}_x\left( \psi (X_1) \right)\right)\right\}, 
\label{characterizing-equation-def:1}\\
&\phi(x) =  \hat q_xf(x)  + (1-\hat q_x)\mathbb{E}_x\left(\phi(X_1)\right), \mbox{ and }\notag \\
&\psi(x)  = \hat q_xh(x)  + (1-\hat q_x)\mathbb{E}_x\left(\psi(X_1)\right) \notag\\
& \mbox{ where $\hat q_x$ is
the maximal 
constant in the set of maximizers in \eqref{characterizing-equation-def:1}.}
\label{characterizing-equation-def:4}
\end{align}
\end{definition}

\begin{theorem}[Verification] \label{ver-thm} Suppose two functions $\phi$ and $\psi$ constitute a solution to the characterizing equation. 
Then, any vector $\hat {\textbf{q}}=\left({\hat {q}_x}\right)_{x\in E}$, with ${\hat {q}}_x$ defined in \eqref{characterizing-equation-def:4},  
is an equilibrium whose equilibrium value function is given by
\begin{align}	\notag
J_{\hat {\textbf{q}}}(x) = \phi(x) + g\left(\psi(x)\right).
\end{align}
 
\end{theorem}
\begin{proof} 
Note that if $x$ is an absorbing state then the expression to be maximized in \eqref{characterizing-equation-def:1} is independent of $q$ and hence $\hat q_x=1$. The result follows directy from Lemma \ref{lem-phi} and Proposition \ref{equivalent-eq-def}. 
 \end{proof}

In the rest of this section we suppose $g$ in \eqref{payoff} is either convex or concave. We remark that the mean-variance problem studied in Section \ref{mean-var-prob} uses a convex $g$ while the variance problem studied in Section \ref {var-prob} uses a concave $g$.

\begin{corollary} \label{cor:NecSufCondConvex} Suppose $g$ in \eqref{payoff} is convex. Then, a stopping strategy ${\hat {\textbf{p}}}$ is an equilibrium if and only if, for each $x\in E$, \eqref{I} and \eqref{II} hold and (at least) one of them holds with equality. 
\end{corollary}
\begin{proof}
If, for each $x$, the necessary conditions \eqref{I} and \eqref{II} hold and one of them holds with equality then ${\hat {\textbf{p}}}$ is a equilibrium since \eqref{Eq''} is then satisfied, for each $x$, with $q=0$ or $q=1$; to see this use that convexity of $g$ implies (Lemma \ref{app-lemma1}) that 
\begin{align}	
q \mapsto K(x,q,{\hat{\textbf{p}}}) \mbox{ is convex}, \label{K-convex}
\end{align}	
and that it generally holds (Lemma \ref{phi-p-expressed-with-Exp}) that
\begin{align}	
\phi_{ {\textbf{p}}}(x) + g\left(\psi_{ {\textbf{p}}}(x)\right)  = K(x,p_x,{{\textbf{p}}}), 
\mbox{ for any $\textbf{p}$ and $x$}.\label{help1}
\end{align}
Let us show the reverse implication: 
If ${\hat {\textbf{p}}}$ is an equilibrium then trivially \eqref{I} and \eqref{II} hold. 
Moreover, by \eqref{K-convex} it 
holds that the maximum in \eqref{Eq''} is attained by either $q=0$ or $q=1$. Recall that if ${\hat {\textbf{p}}}$ is an equilibrium then \eqref{Eq''} holds. 
Now note that  
(i) if $q=0$ is the maximizer in \eqref{Eq''} then \eqref{II} holds with equality, and 
(ii) if $q=1$ is the maximizer in \eqref{Eq''} then \eqref{I} holds equality. 
\end{proof}

\begin{corollary} \label{cor:NecSufCondConcave} Suppose $g$ in \eqref{payoff} is concave and differentiable. Then, 

\begin{itemize}

\item ${\hat {\textbf{p}}}$ is an equilibrium if and only if, for each $x\in E$, \eqref{I}, \eqref{II} and \eqref{III} hold and (at least) one of them holds with equality,

\item if, for a stopping strategy ${\hat {\textbf{p}}}$, \eqref{III} holds with equality for each $x\in E$, then ${\hat {\textbf{p}}}$ is an equilibrium.

\end{itemize}
\end{corollary}
\begin{proof} 
Let us prove the first result: Theorem \ref{lem:nec-cond} implies that if ${\hat {\textbf{p}}}$ is an equilibrium then, for each $x\in E$, \eqref{I}, \eqref{II} and \eqref{III} hold and (at least) one of them holds with equality. 
To see that the other implication is true use concavity of 
$q \mapsto K(x,q,{\hat{\textbf{p}}})$ 
(Lemma \ref{app-lemma1}) and basic optimization theory to see that 
if, for each $x\in E$, \eqref{I}, \eqref{II} and \eqref{III} hold and (at least) one of them holds with equality then the equilibrium condition \eqref{Eq''} must be satisfied (use also the general observation \eqref{help1}). 
The second results is proved similarly. 
\end{proof}

\begin{definition} For an equilibrium ${\hat {\textbf{p}}}$  we denote by ${\hat {\textbf{P}}}$ the set of equivalent equilibria defined as the set of vectors $\textbf{p} \in [0,1]^N$ such that 
$\textbf{p}$ is an equilibrium satisfying $J_{{\textbf{p}}}(x) = J_{\hat {\textbf{p}}}(x)$ for each $x\in E$.
\end{definition}

{
Considering Corollary \ref{cor:NecSufCondConvex} it seems intuitive that a convex $g$ should correspond to a pure equilibrium. It turns out that this is the case but that $g$ must be either strictly convex or affine. Indeed for affine $g$ the problem is easily seen to be time-consistent and it therefore from Theorem \ref{time-con-thm} follows that a stopping strategy is an equilibrium if and only if it corresponds to an optimal stopping time (in the usual sense), and hence by well-known results from the theory of optimal stopping it holds that if $g$ is affine then we only have to search for equilibria in the class of pure stopping strategies. The precise result for $g$ strictly convex is as follows:} 

\begin{theorem} \label{cor:NecSufCondConvexTWO} 
Suppose $g$ in \eqref{payoff} is  strictly 
 convex and an equilibrium ${\hat{\textbf{p}}}$ exists. Then, an equivalent pure equilibrium exists and such a pure equilibrium can be obtained by changing each $\hat {p}_x\in (0,1)$ (in case they exist) to $1$.
\end{theorem} 

\begin{proof}
 Suppose $y\in E$ is such that $\hat {p}_y\in (0,1)$. Then
\begin{align}\notag
q \mapsto K(y,q,{\hat{\textbf{p}}}) 
\left(=qf(y)  + (1-q)\mathbb{E}_y\left(\phi_{\hat{\textbf{p}}}(X_1)\right) + g\left(qh(y) + (1-q)\mathbb{E}_y\left( \psi_{\hat{\textbf{p}}} (X_1) \right)\right)\right)
\end{align}
has a maximum at $q=\hat {p}_y$. Since this function is convex (Lemma \ref{app-lemma1})  and has an interior maximum (by definition of equilibrium and since $\hat {p}_y\in (0,1)$) it must be a constant function. In particular, using that $g$ is strictly convex and that 
\begin{align}\notag
q \mapsto qf(y)  + (1-q)\mathbb{E}_y\left(\phi_{\hat{\textbf{p}}}(X_1)\right)
\end{align}
is linear, we find that  
\begin{align}\notag
q \mapsto qh(y)  + (1-q)\mathbb{E}_y\left(\psi_{\hat{\textbf{p}}}(X_1)\right)
\end{align}
is constant i.e. 
$h(y) = \mathbb{E}_y\left(\psi_{\hat{\textbf{p}}}(X_1)\right)$, but then it also follows that 
$f(y) = \mathbb{E}_y\left(\phi_{\hat{\textbf{p}}}(X_1)\right)$. Now write
\begin{gather*} 
{\tilde{\textbf{p}}}_x=  
\begin{cases}
	{\hat{\textbf{p}}}_x,																&x \neq y \\
	1,																										&x = y.\\
	\end{cases} \label{mean-var-Pr1}
\end{gather*}
Fix a state $x\in E$. Clearly, $\tau_{\tilde{\textbf{p}}}\leq \tau_{\hat{\textbf{p}}}$ a.s. Using the above and the Markov property we obtain 
\begin{align}
\mathbb{E}_x(f(X_{\tau_{{\hat{\textbf{p}}}}}))
\notag & =  \mathbb{E}_x\left(f(X_{\tau_{\hat{\textbf{p}}}})I_{\{\tau_{\tilde{\textbf{p}}}=\tau_{\hat{\textbf{p}}}\}}\right)+
\mathbb{E}_x\left(f(X_{\tau_{\hat{\textbf{p}}}})I_{\{\tau_{\tilde{\textbf{p}}}<\tau_{\hat{\textbf{p}}}\}}\right)\\
\notag & =  \mathbb{E}_x\left(f(X_{\tau_{\hat{\textbf{p}}}})I_{\{\tau_{\tilde{\textbf{p}}}=\tau_{\hat{\textbf{p}}}\}}\right)+
\mathbb{E}_x\left(\mathbb{E}_{X_{\tau_{\tilde{\textbf{p}}}}}\left( \mathbb{E}_{X_1}\left( f(X_{\tau_{\hat{\textbf{p}}}}) \right) \right)
I_{\{\tau_{\tilde{\textbf{p}}}<\tau_{\hat{\textbf{p}}}\}}\right)\\
\notag & =  \mathbb{E}_x\left(f(X_{\tau_{\hat{\textbf{p}}}})I_{\{\tau_{\tilde{\textbf{p}}}=\tau_{\hat{\textbf{p}}}\}}\right)+
\mathbb{E}_x\left(\mathbb{E}_y\left( \phi_{\hat{\textbf{p}}}(X_1)\right)
I_{\{\tau_{\tilde{\textbf{p}}}<\tau_{\hat{\textbf{p}}}\}}\right)\\
\notag & =  \mathbb{E}_x\left(f(X_{\tau_{\hat{\textbf{p}}}})I_{\{\tau_{\tilde{\textbf{p}}}=\tau_{\hat{\textbf{p}}}\}}\right)+
\mathbb{E}_x\left(f(y)I_{\{\tau_{\tilde{\textbf{p}}}<\tau_{\hat{\textbf{p}}}\}}\right)\\
\notag & =  \mathbb{E}_x(f(X_{\tau_{\tilde{\textbf{p}}}})).
	\end{align}
It can similarly be shown that 
\[\mathbb{E}_x(h(X_{\tau_{\hat{\textbf{p}}}})) =  \mathbb{E}_x(h(X_{\tau_{\tilde{\textbf{p}}}})).\]  
It can now be directly verified that ${\tilde{\textbf{p}}}$ is an equilibrium. It it also easy to see that $J_{\hat{\textbf{p}}} (x) = J_{\tilde{\textbf{p}}} (x)$. Now, the claim holds by a trivial induction. 
\end{proof}

{
The following example regards a non-strictly convex $g$ and a mixed equilibrium for which no pure equivalent equilibrium exists; implying that the assumption of \emph{strict} convexity in Theorem \ref{cor:NecSufCondConvexTWO} is necessary.} 

\begin{example} \label{new-example}
{
Consider the Markov chain $X$  defined in Figure \ref{MC-1}. 
\begin{figure}[h]
\centering 
\includegraphics{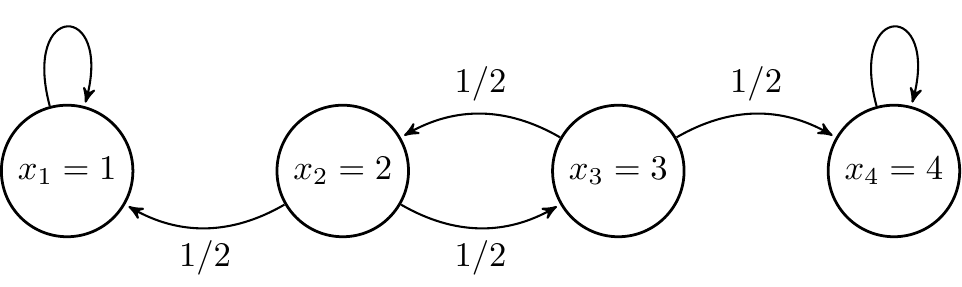}
\caption{The Markov chain $X$ in Example \ref{new-example}.}\label{MC-1}
\end{figure} 
Let $f$ be identically equal to zero, 
$h$ be defined by 
$
h(1)=h(2)=0, 
h(3)=1$ and 
$ 
h(4)=2
$, and $g(x):=(x-1)_+$. Let us first show that 
\begin{align}
\hat {\textbf{p}}= \left(1,\frac{1}{2},0,1\right)^T \label{new-example-eq}
\end{align}
is an equilibrium. 
Since $f=0$ it directly follows, from \eqref{phi-psi-notation}, that 
$\phi_{\hat {\textbf{p}}}(1)
=\phi_{\hat {\textbf{p}}}(2)
=\phi_{\hat {\textbf{p}}}(3)
=\phi_{\hat {\textbf{p}}}(4)=0$.  
Simple calculations also yield  
$
\psi_{\hat {\textbf{p}}}(1)=0, 
\psi_{\hat {\textbf{p}}}(2)=\frac{2}{7}, 
\psi_{\hat {\textbf{p}}}(3)=\frac{8}{7}, 
\psi_{\hat {\textbf{p}}}(4)=2$. This implies that
$\mathbb{E}_{2}\left( \psi_{\hat {\textbf{p}}} (X_1) \right)=\frac{4}{7}$. 
Using the above and  \eqref{K-func} we find that 
\[
K(2,q,{\hat {\textbf{p}}})= g\left(qh(2) + (1-q)\mathbb{E}_{2}\left( \psi_{\hat {\textbf{p}}} (X_1) \right) \right)= \left((1-q)\frac{4}{7}-1\right)_+  
\]
for which $q=\frac{1}{2}=\hat p_2$ is a maximizer (along with every other $q\in[0,1]$); meaning that the condition in \eqref{Eq'''} holds for the state $x_2=2$. Now, find that 
$\mathbb{E}_{3}\left( \psi_{\hat {\textbf{p}}} (X_1) \right)=\frac{8}{7}$, so that  
\[
K(3,q,{\hat {\textbf{p}}})
= g\left(qh(3) + (1-q)\mathbb{E}_3\left(\psi_{\hat {\textbf{p}}} (X_1) \right) \right)
= \left(q+(1-q)\frac{8}{7}-1\right)_+
\]
which is maximized when  $q=0=\hat p_3$, meaning that the condition in \eqref{Eq'''} holds for the state $x_3=3$. 
Clearly, the condition in \eqref{Eq'''} holds also for the absorbing states $x_1$ and $x_4$. We thus conclude that \eqref{Eq'''} holds and that \eqref{new-example-eq} therefore is an equilibrium. 

Let us now verify that that no pure equilibrium equivalent to \eqref{new-example-eq} exists. Since $x_1$ and $x_4$ are absorbing it suffices to check that neither of 
$
\textbf{p}=(1,0,0,1)^T,
\textbf{p}=(1,0,1,1)^T,
\textbf{p}=(1,1,0,1)^T$ or 
$\textbf{p}=(1,1,1,1)^T$ is an equilibrium equivalent to \eqref{new-example-eq}; we remark however, that it is easy to verify that some of these strategies are notwithstanding equilibria. First, note that 
$
J_{\hat {\textbf{p}}}(3) = 
g\left( \psi_{\hat {\textbf{p}}}(3) \right)= \left( \frac{8}{7}   -1\right)_+ = \frac{1}{7}$.   
Second, if $\textbf{p}=(1,1,0,1)^T$ then $\psi_{{\textbf{p}}}(2)=0$ and  $\psi_{{\textbf{p}}}(4)=2$, which implies that
$\psi_{{\textbf{p}}}(3)=\frac{1}{2}\cdot 0+\frac{1}{2}\cdot 2=1$. Hence, $J_{{\textbf{p}}}(3) = 
g\left( \psi_{{\textbf{p}}}(3) \right)= \left(1-1\right)_+ \neq J_{\hat {\textbf{p}}}(3)$ so that 
$\textbf{p}=(1,1,0,1)^T$ cannot be an equilibrium equivalent to \eqref{new-example-eq}. 
Third, if $\textbf{p}=(1,0,0,1)^T$ then it is easy to verify that $\psi_{\textbf{p}}(3) = \frac{4}{3}$, implying that 
$J_{{\textbf{p}}}(3) = 
g\left( \psi_{{\textbf{p}}}(3) \right)= \left( \frac{4}{3}-1\right)_+ = \frac{1}{3} \neq J_{\hat {\textbf{p}}}(3)$ so that 
$\textbf{p}=(1,0,0,1)^T$ cannot be an equilibrium equivalent to \eqref{new-example-eq}. 
Four, if $\textbf{p}=(1,1,1,1)^T$ or $\textbf{p}=(1,0,1,1)^T$ then $\psi_{\textbf{p}}(3) = 1$, implying that 
$J_{{\textbf{p}}}(3) = 
g\left( \psi_{{\textbf{p}}}(3) \right)= \left(1-1\right)_+ \neq J_{\hat {\textbf{p}}}(3)$ so that 
$\textbf{p}=(1,1,1,1)^T$ and $\textbf{p}=(1,0,1,1)^T$ cannot be equilibria equivalent to \eqref{new-example-eq}.}\end{example}

\begin{remark}
The previous theorem implies that in the strictly convex case we just have to check at most the $2^N$ pure strategies to check whether equilibria exist.
\end{remark}

\section{A fixed point problem characterization and existence results} \label{fixed-point-sec}
In this section we derive equilibrium existence results which rely on the observation that an equilibrium is the solution to a certain fixed point problem. 

\begin{definition} \label{Gamma-p-defNEWa} Let $\Gamma:[0,1]^N \rightarrow 2^{[0,1]^N}$ be the point-to-set mapping taking vectors 
${\textbf{p}}\in [0,1]^N$ as input and as output giving $\Gamma({\textbf{p}})$ defined as the set of all vectors 
${\textbf{p}}^*$ which, for each $x \in E$, satisfy
\begin{align} 
{p}_x^*  \in \mbox{\normalfont arg}\hspace{-1mm}\max_{q \in [0,1]}K(x,q,{\textbf{p}}).\label{Gamma-p-def}
\end{align}
\end{definition}

\begin{proposition} \label{FP-thm1} A stopping strategy ${\hat {\textbf{p}}}$ is an equilibrium if and only if it is a fixed point of the mapping $\Gamma$, i.e. if and only if ${\hat {\textbf{p}}} \in \Gamma({\hat {\textbf{p}}})$.
\end{proposition}
\begin{proof} Follows directly from Proposition \ref{equivalent-eq-def}.
\end{proof}

\begin{theorem} \label{FP-thm2} 
Suppose the set of maximizers  in \eqref{Gamma-p-def} is {an interval} for each 
${\textbf{p}}\in [0,1]^N $ and $x \in E$. Then the mapping $\Gamma$ has a fixed point and an equilibrium exists. 
\end{theorem}

\begin{proof} 
The assumed convexity for the set of maximizers for each particular $x$ in \eqref{Gamma-p-def} implies that $\Gamma({\textbf{p}})$ will be a hyperrectangle and thus a convex set. 
Since a maximizer in \eqref{Gamma-p-def} necessarily exists follows that $\Gamma({\textbf{p}})$ is non-empty. To summarize:
\begin{align} 
\mbox{$\Gamma({\textbf{p}})$ is a convex and non-empty set.}\label{kak:1} 
\end{align} 
Suppose $\{_k{\textbf{p}}\}$ is a sequence of vectors in $[0,1]^N$ with $\lim_{k\rightarrow \infty}{_k\textbf{p}}={\textbf{p}}$. From Lemma \ref{lem1} we know that 
\begin{align} \notag
\lim_{k\rightarrow \infty}\mathbb{E}_x\left(\phi_{_k{\textbf{p}}}(X_1)\right) = \mathbb{E}_x\left(\phi_{\textbf{p}}(X_1)\right), 
\mbox{ for each $x\in E$}. \label{FP-thm2:eq} 
\end{align} 
Using the analogous result for $\psi_{\textbf{p}}$ and Assumption \ref{g-assum} we find that for any fixed $q$ and $x$ it holds that 
\[
\lim_{k\rightarrow \infty}K(x,q,{_k{\textbf{p}}}) = K(x,q,{{\textbf{p}}}).
\]
Hence, using also that $q\mapsto K(x,q,{\textbf{p}})$ is continous (for any fixed $x$ and ${\textbf{p}}$), it is easy to see that: 
\begin{align} 
\label{kak:2}
\begin{split} 
&\mbox{For any two sequences 
$\{_k{\textbf{p}}\}$ and $\{_k{\textbf{q}}\}$ on $[0,1]^N$, with $\lim_{n\rightarrow \infty}{_k\textbf{p}}={\textbf{p}}$ and }\\
& \mbox{$\lim_{n\rightarrow \infty}{_k\textbf{q}}={\textbf{q}}$, that satisfy
$_k{\textbf{q}} \in \Gamma(_k{\textbf{p}})$ for all $k$,  it holds that ${\textbf{q}}\in \Gamma({\textbf{p}})$.}
\end{split} 
\end{align} 

We also note that:
\begin{align}
\mbox{The set $[0,1]^N$ is non-empty, compact and convex.} \label{kak:3}
\end{align}

From \eqref{kak:1},\eqref{kak:2} and \eqref{kak:3} 
it follows that 
we may use Kakutani's fixed point theorem, see e.g. \cite[p. 121]{fan1952fixed}, to conclude that the mapping $\Gamma$ has a fixed point. The existence of an equilibrium follows (using also Proposition \ref{FP-thm1}). 
\end{proof}

\begin{corollary} \label{g-concave-eq-exists-cor} Suppose $g$  in \eqref{payoff} is concave, then an equilibrium exists.
\end{corollary}
\begin{proof}  
The concavity of $g$ implies the concavity of $q \mapsto K(x,q,{\textbf{p}})$, see Lemma \ref{app-lemma1}. It directly follows that the set of maximizers  in \eqref{Gamma-p-def} is an interval for each ${\textbf{p}}\in [0,1]^N $ and $x \in E$. The result follows from Theorem \ref{FP-thm2}. 
\end{proof}

\begin{remark}
{
The property that an equilibrium is a fixed point of some suitably defined mapping and the use of fixed point theorems to establish existence of equilibria is standard in game theory. An early reference observing this and relying on Kakutani's fixed point theorem is \cite{nash1950equilibrium}. The connection to fixed point problems has also been made in time-inconsistent control theory.} 
In \cite{huang2018strong} time-inconsistent regular stochastic control in continuous time is studied and an equilibrium existence result is proved using fixed point arguments similar to those used here. 
In \cite{tomas-disc} time-inconsistent stochastic control in discrete time is studied and it is noted that an equilibrium can be viewed as the fixed point of a particular mapping. 
\end{remark}

\section{The myopic adjustment process and equilibrium stability} \label{sec:myopic}

An iteration of the type ${\textbf{p}}_{k+1}\in \Gamma({\textbf{p}}_k)$ where ${\textbf{p}}_0 \in [0,1]^N$ (see Definition \ref{Gamma-p-defNEWa}) corresponds to what in economics is known as a myopic adjustment process for decisions in repeated interactive situations, 
see e.g. \cite{
borgers1997learning,
erev1998predicting,
kosfeld2002myopic}. The interpretation here is that every agent in the game adjusts his decision at each $k$-step under the (myopic) assumption that all other agents will stay with their strategy. 
Since $\Gamma({\textbf{p}}_k)$ is in general a set of vectors, i.e. a set of stopping strategies, it holds that this iteration is not uniquely defined. We thus define ${{\bar \Gamma}}({\textbf{p}}_k)$ as the largest (in Euclidean norm, or, equivalently, element-wise) vector in $\Gamma({\textbf{p}}_k)$ 
 and consider the myopic adjustment process
\begin{align} \label{myopic-adj-proc}
{\textbf{p}}_{k+1}= {{\bar \Gamma}}({\textbf{p}}_k), \mbox{ with }  {\textbf{p}}_0 \in [0,1]^N.
\end{align}
This corresponds to the interpretation that there is preference in the myopic adjustment for a higher probability of stopping over a smaller one when they give the same value (in the maximization in \eqref{Gamma-p-def}). Note that the myopic adjustment process can be tried as a constructive algorithm for finding equilibria.

The following are now natural questions:

\textbf{1.} Suppose ${\hat {\textbf{p}}}$ is an equilibrium and that we perturb ${\hat {\textbf{p}}}$ slightly by considering a stopping strategy 
${\hat {\textbf{p}}}^{\epsilon} \in B({\hat {\textbf{p}}};\epsilon)$ for some small $\epsilon>0$, where $B({\hat {\textbf{p}}};\epsilon)$ denotes a ball with radius $\epsilon$ centered at ${\hat {\textbf{p}}}$. In which circumstances does then the myopic adjustment process \eqref{myopic-adj-proc} with ${\textbf{p}}_0={\hat {\textbf{p}}}^\epsilon$ converge to ${\hat {\textbf{p}}}$?

\textbf{2.} In which circumstances does the myopic adjustment process \eqref{myopic-adj-proc} converge to an equilibrium ${\hat {\textbf{p}}}$ for \emph{any}  initial value ${\textbf{p}}_0$?

In the rest of this section we try to shed some light on these questions by defining and investigating different notions of equilibrium stability.

\begin{definition} An equilibrium ${\hat {\textbf{p}}}$ is said to be \emph{strongly locally stable} if there exists a constant $\epsilon$ such that for every 
${{\hat {\textbf{p}}}^{\epsilon}}\in B({\hat {\textbf{p}}};\epsilon)$ there exists an equivalent equilibrium $\textbf{p} \in {\hat {\textbf{P}}}$ such that for every $x$ it holds that
\begin{align}\notag
K(x,{p}_x,{{\hat {p}}^{\epsilon}}) \geq K(x,q,{{\hat {\textbf{p}}}^{\epsilon}}), \mbox{ for all $q \in [0,1]$.}
\end{align}
\end{definition}

\begin{definition} An equilibrium ${\hat {\textbf{p}}}$ is said to be \emph{locally stable} if for some $\epsilon>0$ and every ${{\hat {\textbf{p}}}^{\epsilon}} \in B({\hat {\textbf{p}}};\epsilon)$ the myopic adjustment process with ${\textbf{p}}_0 = {{\hat {\textbf{p}}}^{\epsilon}}$ converges to an equivalent equilibrium $\textbf{p} \in {\hat {\textbf{P}}}$. An equilibrium is said to be \emph{unstable} if it is not locally stable. 
\end{definition}

It is easy to see that a strongly locally stable equilibrium is locally stable.

\begin{remark} 
 The interpretation of a strongly locally stable equilibrium 
is that the best response (at each $x$) to a small deviation from the equilibrium is to return to an equivalent equilibrium 
immediately; i.e., if a small deviation from a strongly locally stable equilibrium occurs then 
the myopic adjustment process 
converges in one step to an equivalent equilibrium. The interpretation of a locally stable equilibrium is that if a small deviation from the equilibrium occurs then the equilibrium (or more precisely an equivalent equilibrium) will eventually be restored under the myopic adjustment process.  
\end{remark}

 We obtain:

\begin{theorem} \label{strongly-locally-stable-thm} Suppose $g$ is 
strictly 
 convex and that an equilibrium  ${\hat {\textbf{p}}}$  exists. Then,  ${\hat {\textbf{p}}}$ is strongly locally stable.
\end{theorem}
\begin{proof} In this proof let us use the notation 
\[{\tilde{\textbf{p}}} = {{\bar \Gamma}}({{\hat {\textbf{p}}}^{\epsilon}}).\]
Now, if we can show that ${\tilde{\textbf{p}}}$ is an equilibrium equivalent to  ${\hat {\textbf{p}}}$ for some small $\epsilon>0$ then we are done. 

Fix an arbitrary state $x\in E$. Since $q \mapsto K(x,q,{\hat{\textbf{p}}})$ is a convex function it follows that exactly one of the following cases holds:  

\emph{Case $1$: $\mbox{\normalfont arg}\hspace{-1mm}\max_{q \in [0,1]} K(x,q,{\hat {\textbf{p}}})=\{0\}$.} 
This means that ${\hat {p}}_x=0$ and 
$K(x,0,{\hat {\textbf{p}}})>K(x,q,{\hat {\textbf{p}}})$ for all $q\in(0,1]$. 
Now use that $K(x,q,{\textbf{p}})$ is continuous in ${\textbf{p}}$ (Lemma \ref{lem1}), and convex in $q$ (Lemma \ref{app-lemma1}) to see that there exists an $\epsilon=\epsilon_x>0$ such that 
\[
\mbox{\normalfont arg}\hspace{-1mm}\max_{q \in [0,1]} K(x,q,{{\hat {\textbf{p}}}^{\epsilon}})=\{0\},\]
i.e. ${\tilde{p}}_x={\hat{p}}_x$. 

\emph{Case $2$: $\mbox{\normalfont arg}\hspace{-1mm}\max_{q \in [0,1]} K(x,q,{\hat {\textbf{p}}})=\{1\}$.} 
In the same way as above we find $\epsilon=\epsilon_x>0$ such that
${\tilde{p}}_x= {\hat{p}}_x$. 

\emph{Case $3$: $\mbox{\normalfont arg}\hspace{-1mm}\max_{q \in [0,1]} K(x,q,{\hat {\textbf{p}}})=[0,1]$.}
This means that $K(x,q_1,{\hat {\textbf{p}}})=K(x,q_2,{\hat {\textbf{p}}})$ for all $q_1,q_2\in[0,1]$. No matter how $\epsilon$ is chosen, convexity trivially implies that
\[{\tilde{p}}_x \in \{0,1\}.\]

Summarizing the three cases, for a sufficiently small $\epsilon$ we conclude that ${\tilde{\textbf{p}}}$ satisfies 
\begin{gather*} 
{\tilde{p}}_x=  
\begin{cases}
	{\hat{p}}_x,						 & 
	\mbox{when } \mbox{\normalfont arg}\hspace{-1mm}\max_{q \in [0,1]} K(x,q,{\hat {\textbf{p}}})=\{0\}\mbox{ or }=\{1\},\\
	1 \mbox{ or } 0,															 & 
	\mbox{when } \mbox{\normalfont arg}\hspace{-1mm}\max_{q \in [0,1]} K(x,q,{\hat {\textbf{p}}})=[0,1].\\
	\end{cases} \label{mean-var-Pr}
\end{gather*}
We now show that ${\tilde{\textbf{p}}}$ is an equilibrium equivalent to  ${\hat {\textbf{p}}}$. Indeed, 
note that the condition $\mbox{\normalfont arg}\hspace{-1mm}\max_{q \in [0,1]} K(x,q,{\hat {\textbf{p}}})=[0,1]$ means that $q\mapsto K(x,q,{\hat {\textbf{p}}})$ is a constant function. Hence, using the exact same arguments as in the proof of Theorem \ref{cor:NecSufCondConvexTWO}, we see that $J_{\tilde{\textbf{p}}} = J_{\hat {\textbf{p}}}$, proving the claim.
\end{proof} 
Theorem \ref{strongly-locally-stable-thm} implies that the equilibria in the mean-variance problems studied in Section \ref{mean-var-prob} below are locally stable.

\begin{definition} An equilibrium ${\hat {\textbf{p}}}$ is said to be \emph{globally stable} if the myopic adjustment process converges to an equivalent equilibrium $\textbf{p} \in {\hat {\textbf{P}}}$ for any starting value ${{\textbf{p}}}_0$.
\end{definition}

Obviously, a globally stable equilibrium is a unique equilibrium and a globally stable equilibrium is also locally stable. However, a globally stable equilibrium is not generally strongly locally stable. 

A strictly convex function $g$ makes the problem of checking strong stability a finite problem as only pure stopping strategies have to be considered. With this notion, we may analyze the structure using the notion of directed graphs: The vertices are the pure strategies $\textbf{p}\in\{0,1\}^N$ and there is a directed edge from $\textbf{p}$ to $\textbf{q}$ if 
$	\textbf{q}= {{\bar \Gamma}}({\textbf{p}}).$ Now, the problem of studying strong stability boils down to checking whether this directed graph is acyclic.

We remark that Example \ref{ex-2-eq} is a problem with a{
strictly} convex $g$ and two equilibria and hence{
strict} convexity of $g$ is not a sufficient condition for global stability.

We immediately obtain the following (trivial) result:

\begin{theorem} 
Suppose $g$ is strictly convex. Then: Either the myopic procedure does not converge but runs in cycles or it terminates in at most 
	$2^N$
	 steps. 
\end{theorem}

\section{Applications} \label{examples}
In this section we apply the developed theory to mean-variance and variance optimization problems.

\subsection{A mean-variance problem} \label{mean-var-prob}
The mean-variance stopping problem --- see Section \ref{rel-lit} for a motivation --- is attained in the framework of the present paper when
\begin{align}
f(x)= -\gamma x^2, \enskip g(x)= x +\gamma x^2 \mbox{ and } h(x)=x, \mbox{ with $\gamma>0$.}
\label{mean-var-f-h-g}
\end{align}
The{
strict} convexity of $g$ implies that \emph{if} an equilibrium exists then a pure version of that equilibrium exists, cf. Theorem \ref{cor:NecSufCondConvexTWO}, and that it is moreover strongly locally stable, cf. Theorem \ref{strongly-locally-stable-thm}. 
\subsubsection{Equilibrium strategies of threshold-type}\label{subsubsec:threshold}
In \cite[Section 4.2]{christensen2018time} it was shown that the equilibrium for the mean-variance problem for a geometric Brownian motion corresponds, in case it exists, to using a particular threshold stopping strategy.  
In this section we first study a threshold strategy ansatz to finding an equilibrium stopping time for the mean-variance problem assuming only that $X$ is a 
skip free Markov chain 
absorbed in $x_1$ and $x_N$ on some state space  $E = \{x_1,x_2,...,x_N\}$ 
 with $\mathbb{P}_{x_i}(X_1=x_{i+1})=1/2=\mathbb{P}_{x_i}(X_1=x_{i-1})$ for all $i=2,...,N-1$. We furthermore assume 
that 
 $x_1=0$ (this makes some expressions shorter and can be easily relaxed). Second, we use this ansatz to a more particular problem. 

An (upper) threshold stopping time
\begin{align}\notag
\tau_{\textbf{p}}= \min\{n\geq 0: X_n \geq x_b\},
\end{align}
is easily seen to be attained by the stopping strategy
\begin{align}
\label{threshold-p}
\begin{split}
& \textbf{p}= \left(1,p_{x_2},...,p_{x_{N-1}},1\right)^T, \mbox{ with $p_{x_i}=0$ for $i\in \{2,...,b-1\}$  and}\\
& \mbox{$p_{x_i}=1$  for $i\in \{b,...,N-1\}$.} 
\end{split}
\end{align}
(Note that the values of $p_{x_1}$ and $p_{x_N}$ are irrelevant since $x_1$ and $x_N$ are absorbing states.) 
The convexity of $g$ and Corollary \ref{cor:NecSufCondConvex} imply that ${\textbf{p}}$ in \eqref{threshold-p} is an equilibrium if and only if, for all $i\in \{1,...,N\}$,
\begin{align}	
\notag\phi_{\textbf{p}}(x_i) + g\left(\psi_{\textbf{p}}({x_i})\right)  & \geq f({x_i}) + g(h({x_i})),\\
\notag\phi_{\textbf{p}}(x_i) + g\left(\psi_{\textbf{p}}({x_i})\right)  & \geq  
\mathbb{E}_{x_i}\left(   \phi_{\textbf{p}}(X_1)  \right) + g\left(\mathbb{E}_{x_i}\left( \psi_{\textbf{p}}(X_1) \right)\right).
\end{align}
To see this note e.g. that ${\textbf{p}}$ being pure implies that one of these conditions necessarily holds.
This implies that ${\textbf{p}}$ in \eqref{threshold-p} is an equilibrium if and only: 
\begin{align}	
 & \mathbb{E}_{x_i}\left(   \phi_{\textbf{p}}(X_1)  \right) + g\left(\mathbb{E}_{x_i}\left( \psi_{\textbf{p}}(X_1) \right)\right)\geq f({x_i}) + g(h({x_i})), 
\mbox{ for $i \in \{2,...,b-1\}$}, \label{mean-var-eqcond1}\\
& f({x_i}) + g(h({x_i}))  \geq  \mathbb{E}_{x_i}\left(   \phi_{\textbf{p}}(X_1)  \right) + g\left(\mathbb{E}_{x_i}\left( \psi_{\textbf{p}}(X_1)\right)\right), 
\mbox{ for $i \in \{b,...,N-1\}$}. \label{mean-var-eqcond2}
\end{align} 
To see this use the threshold structure of ${\textbf{p}}$ and that $x_1$ and $x_N$ are absorbing. 
Now consider the function
\begin{align}
H({x_i},b):= \mathbb{E}_{x_i}\left(   \phi_{\textbf{p}}(X_1)  \right) + g\left(\mathbb{E}_{x_i}\left( \psi_{\textbf{p}}(X_1) \right)\right) 
- f({x_i}) - g(h({x_i})).\notag
\end{align}
Since ${\textbf{p}}$ in \eqref{threshold-p} is an equilibrium if and only if \eqref{mean-var-eqcond1} and \eqref{mean-var-eqcond2} hold it follows that
\begin{align}
\begin{split}
& \mbox{${\textbf{p}}$ in \eqref{threshold-p} is an equilibrium if and only if:}\\
& H({x_i},b) \geq 0, \mbox{ for $i \in \{2,...,b-1\}$ and } 
  H({x_i},b) \leq 0, \mbox{ for $i \in \{b,...,N-1\}$}. \label{mean-var-eqcond-H}
\end{split}
\end{align}
Since $x_1$ is absorbing it holds that
\begin{gather*} 
\phi_{\textbf{p}}(x_i) =  
\begin{cases}
	f(x_1) + (f(x_b)-f(x_1))Pr(i,b),																& i \in \{1,...,{b-1}\}\\
	f(x_i),																													& i \in \{b,...,N\},\\
	\end{cases} \label{mean-var-Pr} 
\end{gather*}
for $Pr(i,b):=\mathbb{P}_{x_i}\left(X_{\min\{n\geq 0:X_n \in\{x_1,x_b\}\}}=x_b\right)$ with $X_0=x_i \in \{x_1,...,x_b\}$; 
where $Pr(i,b)$ is determined by the recurrence relation 
$Pr(i,b) = \frac{1}{2}  Pr(i+1,b)  + \frac{1}{2} Pr(i-1,b)$ for $i \in \{2,...,b-1\}$ 
with boundary conditions $Pr(1,b)=0$ and $Pr(b,b)=1$, yielding
\begin{gather*}
Pr(i,b) =  \frac{i-1}{b-1}.
\end{gather*}
Basic probability calculations now yield

$\mathbb{E}_{x_i}\left(   \phi_{\textbf{p}}(X_1)  \right)$ 
\begin{gather*}
= \begin{cases}
	f(x_1),																												  & i =1\\
	\frac{1}{2} \phi_{\textbf{p}}(x_{i+1}) + \frac{1}{2} \phi_{\textbf{p}}(x_{i-1}),			  & i \in \{2,...,N-1\}\\
	f(x_N)  																												& i =N\\
	\end{cases}\quad \quad \quad \quad \quad \quad \quad \quad \quad \quad  \\
	= \begin{cases}
	f(x_1) + (f(x_b)-f(x_1)) \frac{i-1}{b-1} ,	 				  												& i \in \{1,...,b-1\}\\
	\frac{1}{2}f(x_{b+1}) +\frac{1}{2}(f(x_1) + (f(x_b)-f(x_1)) \frac{b-2}{b-1},	 									& i = b\\
	\frac{1}{2}f(x_{i+1}) + \frac{1}{2}f(x_{i-1}),	 																				& i \in \{b+1,...,N-1\}\\
	f(x_N),  																												& i =N.
	\end{cases} 
		\end{gather*}
Using \eqref{mean-var-f-h-g} and $x_1=0$ this implies that
\begin{gather*}
\mathbb{E}_{x_i}\left(   \phi_{\textbf{p}}(X_1)  \right)  
		= \begin{cases}
	-\gamma x_b^2 \frac{i-1}{b-1},	 				 																&  i \in \{1,...,b-1\}\\
	-\frac{1}{2}\gamma x_{b+1}^2 -\frac{1}{2}\gamma x_b^2 \frac{b-2}{b-1},	 	&  i = b\\
	-\frac{1}{2}\gamma x_{i+1}^2 -\frac{1}{2}\gamma x_{i-1}^2,	 			&  i \in \{b+1,...,N-1\}\\
	-\gamma x_N^2,  																									& i =N.
	\end{cases} 
		\end{gather*}

Analogously,
\begin{gather*}
\mathbb{E}_{x_i}\left(   \psi_{\textbf{p}}(X_1)  \right)
	= \begin{cases}
	x_b \frac{i-1}{b-1},	 				  																			&  i \in \{1,...,b-1\}\\
	\frac{1}{2}x_{b+1} +\frac{1}{2}x_b\frac{b-2}{b-1},	 									&  i = b\\
	\frac{1}{2}x_{i+1} + \frac{1}{2}x_{i-1},	 										&  i \in \{b+1,...,N-1\}\\
	x_N,  																												& i =N.
	\end{cases} 
	\end{gather*}
Note also that $- f({x_i}) - g(h({x_i})) = - x_i$. The observations above yield, with some calculations
	\begin{gather*}
	\resizebox{\hsize}{!}{$
 H(x_i,b)=
 \begin{cases}
	x_b \frac{i-1}{b-1}\left(1-\gamma x_b\left(1-\frac{i-1}{b-1}\right)\right) -x_i,	 				 																&  i \in \{2,...,b-1\}\\
	 \frac{x_{b+1}(1-\gamma x_{b+1})}{2} + 
	 \frac{x_{b}\left(1-\gamma x_{b}\right)}{2}\frac{b-2}{b-1} + 
	\gamma\frac{\left(x_{b+1}+x_b\frac{b-2}{b-1}\right)^2}{4} - x_b,	 	&  i = b\\
\frac{x_{i+1}+x_{i-1}}{2} - \frac{\gamma}{4} (x_{i+1}-x_{i-1})^2 -x_i,	 			&  i \in \{b+1,...,N-1\}.
	\end{cases}
$}
	\end{gather*}
Using this explicit formula we can --- for any $N$, $\gamma$ and further specification of the state space $E$ ---  check if \eqref{mean-var-eqcond-H} is satisfied for some $\hat b \in \{2,...,N-1\}$, in which case this $\hat b$ corresponds to an equilibrium. Moreover, if such a $\hat b$ exists then the observations above imply that the corresponding equilibrium value function is
\begin{gather*}
J_{\hat {\textbf{p}}}(x_i) =  \phi_{\hat {\textbf{p}}}(x_i) + g(\psi_{\hat {\textbf{p}}}(x_i)) =
 \begin{cases}  
	x_1=0,																							& i = 1\\    
 H(x_i,\hat b) +x_i																		& i \in \{2,...,b-1\}\\  
	x_i																						  & i \in \{b,...,N\}, \\
 \end{cases}%
\end{gather*}
where ${\hat {\textbf{p}}}$ denotes the threshold strategy, cf. \eqref{threshold-p}, corresponding to $\hat b$.

Let us now consider a specific example. Suppose $\gamma = 0.07$ and that
\begin{align}
E = \{x_1,x_2,...,x_N\}, \mbox{ with $x_1=0$, $x_{i+1} -x_{i}= i/10$ and $N=18$.} 
\label{state-space-mean-var}
\end{align}
The state space $E$ is depicted in Figure \ref{fig-mean-var2}. For this example we conclude from \eqref{mean-var-eqcond-H} and the second picture in Figure \ref{fig-mean-var2} that the threshold strategy \eqref{threshold-p} with $b=\hat b=16$ is an equilibrium. Figure \ref{fig-mean-var2} also depicts the corresponding equilibrium value function together with the value for the strategy of always stopping immediately.

\begin{figure}
\centering 
\includegraphics{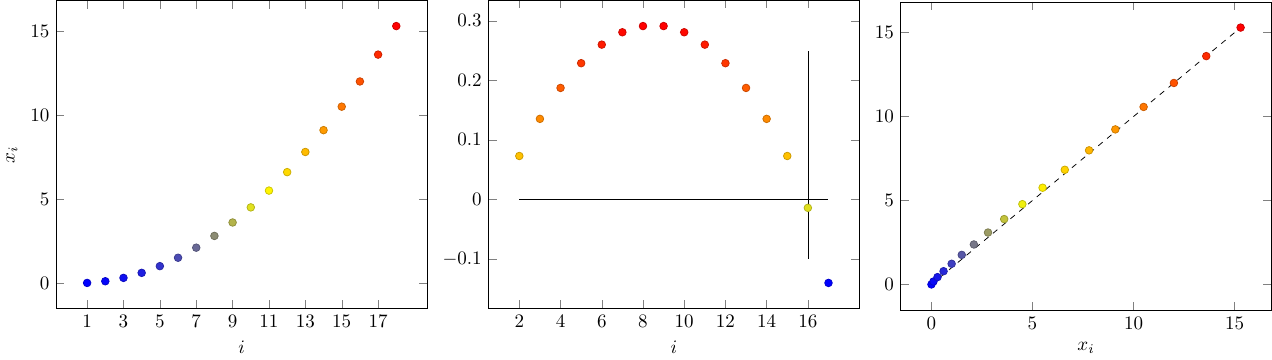}
	\caption{
	First picture: A representation of the state space defined in \eqref{state-space-mean-var} with $N=18$.
	Second picture: $i\mapsto H(x_i,b)$ for $b=16$. 
  Third picture: the equilibrium value function $x_i\mapsto J_{\hat {\textbf{p}}}(x_i)$ (dots) and $x_i\mapsto x_i$ (dashed).}
	\label{fig-mean-var2}
	\end{figure}

\begin{remark}
The equilibrium value function in Figure \ref{fig-mean-var2} looks very similar to the equilibrium value function for the geometric Brownian motion mean-variance stopping problem depicted in \cite[Figure 2]{christensen2018time}. 
\end{remark}

\subsubsection{Counterexamples to uniqueness and existence} \label{uniqueness-existence-examples}
In this section we show that one should not in general expect an equilibrium to be unique, not only in the trivial sense that more than one equilibrium strategy may exist, but also in the sense that these may correspond to different equilibrium value functions. 
We also show that one should not in general expect an equilibrium to exist.

\begin{example} [Two different equilibria] \label{ex-2-eq} 
Consider the mean-variance problem --- i.e. $f,g$ and $h$ as defined in \eqref{mean-var-f-h-g} --- for some $\gamma>2$ and the Markov chain $X$ defined in Figure \ref{MC-2}. 
\begin{figure}[h]
\centering 
\includegraphics{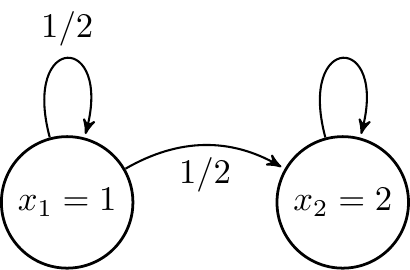}
\caption{The Markov chain $X$ in Example \ref{ex-2-eq}.}\label{MC-2}
\end{figure}

Let us show that $\hat {\textbf{p}}= (1,1)^T$, i.e. the strategy corresponding to always stopping immediately, is an equilibrium. 
Clearly,  
$\phi_{\hat {\textbf{p}}}(x_i) + g(\psi_{\hat {\textbf{p}}}(x_i)) = f(x_i)  + g(h(x_i)) = x_i$ for each $i$. Hence \eqref{I} holds with equality for each $i$. 
Simple calculations give that 
$\mathbb{E}_{x_1}\left(   \phi_{\hat {\textbf{p}}}(X_1)  \right) = \frac{-\gamma }{2}\left(2^2 + 1^2\right) = -\frac{5}{2}\gamma$
and 
$\mathbb{E}_{x_1}\left( \psi_{\hat {\textbf{p}}}(X_1) \right) = \frac{3}{2}$. Hence, 
$$\mathbb{E}_{x_1}\left(   \phi_{\hat {\textbf{p}}}(X_1)  \right) + g\left(\mathbb{E}_{x_1}\left(   \psi_{\hat {\textbf{p}}}(X_1)  \right)\right) = 
\frac{6- \gamma}{4}<1=f(x_1)  + g(h(x_1)).$$ 
Hence, \eqref{II} holds for $i=1$. Moreover, \eqref{II} holds also for $x_2$ since this is an absorbing state. It thus follows from Corollary \ref{cor:NecSufCondConvex} that ${\hat {\textbf{p}}}$ is an equilibrium. The corresponding equilibrium value is easily found to be $J_{\hat {\textbf{p}}}= (1,2)^T$. 
With similar calculations it can be shown that also $\tilde {\textbf{p}}= (0,1)^T$, i.e. waiting until the value $2$ is reached, is an equilibrium  with corresponding equilibrium value function $J_{\tilde {\textbf{p}}}= (2,2)^T$.
\end{example}

\begin{example} [No equilibrium]  \label{ex-0-eq} Consider the mean-variance problem with $\gamma = 1$  for the skip free Markov chain $X$  defined in Figure \ref{MC-3}.
\begin{figure}[h]
\centering 
\includegraphics{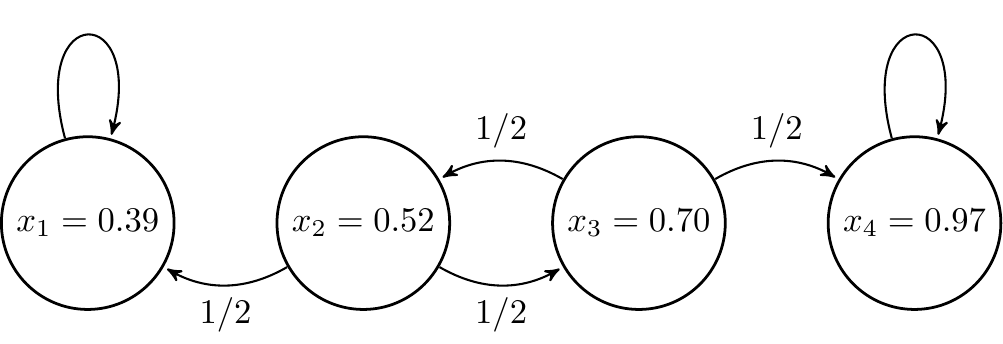}
\caption{The Markov chain $X$ in Example \ref{ex-0-eq}.}\label{MC-3}
\end{figure}

Since $g$ is{
strictly} convex  --- implying that if an equilibrium exists then a pure equilibrium exists, cf. Theorem \ref{cor:NecSufCondConvexTWO}
 --- and $x_1$ and $x_N$ are absorbing it follows that verifying that the strategies 
$(i)$--$(iv)$ 
below are not equilibrium strategies corresponds to verifying that this problem has no equilibrium. 
From the calculations below  and the{
strict} convexity of $g$ it is easy to see that a myopic adjustment process for this problem does not converge but instead --- regardless of the initial value ${\textbf{p}}_0$ --- runs in a cycle according to 
$(i) \rightarrow (ii) \rightarrow (iii) \rightarrow (iv) \rightarrow (i)$
and so on as long as it is allowed to run.  

\begin{enumerate}
\item[\rm(i)] \label{ex-0-eq:i} 
${\textbf{p}}= (1,1,1,1)^T$: 
First, $\phi_{{\textbf{p}}}(x_2) + g(\psi_{{\textbf{p}}}(x_2)) = f(x_2)  + g(h(x_2)) = x_2 = 0.52$. 
Second, 
$\mathbb{E}_{x_2}\left(\phi_{{\textbf{p}}}(X_1)\right) = \frac{1}{2}f(x_1) + \frac{1}{2}f(x_3) \approx -0.3211$ and  
$\mathbb{E}_{x_2}\left(\psi_{{\textbf{p}}}(X_1)\right) =  \frac{1}{2}h(x_1) + \frac{1}{2}h(x_3)  \approx  0.5450$; which gives 
$\mathbb{E}_{x_2}\left(\phi_{{\textbf{p}}}(X_1)\right) + g\left(\mathbb{E}_{x_2}\left(\psi_{{\textbf{p}}}(X_1)\right)\right) 
\approx  0.5210$. 
Hence, deviating by not stopping is optimal at $x=2$ and this is therefore not an equilibrium (cf. e.g. Theorem \ref{lem:nec-cond}).

\item[\rm(ii)]  \label{ex-0-eq:ii}  ${\textbf{p}}= (1,0,1,1)^T$: 
First, $\phi_{{\textbf{p}}}(x_3) + g(\psi_{{\textbf{p}}}(x_3)) = f(x_3)  + g(h(x_3)) = x_3 =  0.70$.
Second, $\mathbb{E}_{x_3}\left(\phi_{{\textbf{p}}}(X_1)\right) 
= \frac{1}{2}f(x_4) + \frac{1}{2}\left(\frac{1}{2}f(x_1) + \frac{1}{2}f(x_3)\right) \approx  -0.6310$ and  
$\mathbb{E}_{x_3}\left(\psi_{{\textbf{p}}}(X_1)\right) 
= \frac{1}{2}h(x_4) + \frac{1}{2}\left(\frac{1}{2}h(x_1) + \frac{1}{2}h(x_3)\right) 
\approx   0.7575$; which gives 
\[\mathbb{E}_{x_3}\left(\phi_{{\textbf{p}}}(X_1)\right)+ g\left(\mathbb{E}_{x_2}\left(\psi_{{\textbf{p}}}(X_1)\right)\right)
\approx   0.7003.\] 
Hence, this is not an equilibrium.

\item[\rm(iii)]   \label{ex-0-eq:iii}  ${\textbf{p}}= (1,0,0,1)^T$: 
First, $\phi_{{\textbf{p}}}(x_2)= Pr(2,4) f(x_4) + (1-Pr(2,4))f(x_1) \approx -0.4150$ 
(where $Pr(2,4)=\frac{1}{3}$, cf. Section \ref{subsubsec:threshold}) and  
\[\psi_{{\textbf{p}}}(x_2) = Pr(2,4) h(x_4) + (1-Pr(2,4))h(x_1) \approx 0.5833.\] 
Hence, $\phi_{{\textbf{p}}}(x_2) + g(\psi_{{\textbf{p}}}(x_2)) \approx    0.5086$. 
Second, $f(x_2)  + g(h(x_2)) = x_2 = 0.52$. Hence, this is not an equilibrium.

\item[\rm(iv)]  \label{ex-0-eq:iv} ${\textbf{p}}= (1,1,0,1)^T$: 
First, $\phi_{{\textbf{p}}}(x_3) =\frac{1}{2}f(x_2)+ \frac{1}{2}f(x_4)\approx  -0.6057$ 
and $\psi_{{\textbf{p}}}(x_3) = \frac{1}{2}h(x_2)+ \frac{1}{2}h(x_4) \approx   0.7450$. 
Hence, $\phi_{{\textbf{p}}}(x_3) + g(\psi_{{\textbf{p}}}(x_3)) \approx 0.6944$. 
Second, $f(x_3)  + g(h(x_3)) = x_3 = 0.70$. Hence, this is not an equilibrium.

\end{enumerate}

\subsection{A variance problem} \label{var-prob}
The variance problem is defined by setting $J_\tau(X_\tau)= \mbox{Var}_x(X_\tau)$ in \eqref{payoff} which in our framework is attained when
\begin{align}
f(x):=  x^2, \enskip  g(x)=- x^2 \mbox{ and } h(x)=x. \label{var-f-h-g}
\end{align}
The equilibrium approach to the variance stopping problem for a geometric Brownian motion was studied in \cite[Section 4.1]{christensen2018time}. Optimal variance problems are also studied in e.g. 
\cite{
buonaguidi2015remark,
buonaguidi2018some,
gad2019optimal,
gad2015variance,
pedersen2011explicit}. 

From Corollary \ref{g-concave-eq-exists-cor} and the concavity of $g$ it follows that an equilibrium always exists for the variance problem. 
Let us now consider a symmetric random walk $X$ on the state space
\begin{align}\notag
E = \{x_0,x_1,...,x_{M-1},x_M\} = \{0,1,...,M-1,M\},
\end{align}
where $x_0$ is absorbing and $x_M$ is reflecting, for some natural number $M$. 
Note that the number of states is $N=M+1$. Let us try the ansatz that the equilibrium is of the kind 
\begin{align}
& \textbf{p}= (1,0,...,0,p)^T \label{strat-var}
\end{align}
for some $p\in [0,1]$ to be determined. Similarly to Section \ref{mean-var-prob} we find that
\begin{align}
Pr(i):=\mathbb{P}_{x_i}\left(X_{\min\{n\geq 0:X_n \in\{x_0,x_M\}\}}=x_M\right) = \frac{i}{M}.\notag
\end{align}
Using that $p$ is the probability of stopping at $x_M$, and also 
\eqref{var-f-h-g}, \eqref{strat-var} and that $x_0=0$ is absorbing, we find  
\begin{align}
\phi_{\textbf{p}}\left(x_M\right)= 
p M^2 + (1-p) Pr\left(M-1\right)\phi_{\textbf{p}}\left(x_M\right).\notag
\end{align}
This implies that
\begin{align}
\notag \phi_{\textbf{p}}\left(x_M\right) 
& = \frac{ p M^2}{1- (1-p) Pr\left(M-1 \right)}\\
\notag & = \frac{ p M^3}{M- (1-p) (M-1)}.
\end{align}
Since $x_M$ is reflecting it follows that
\begin{align}\notag
\mathbb{E}_{x_M}\left(   \phi_{\textbf{p}}(X_1)  \right) = \phi_{\textbf{p}}\left(x_{M-1}\right).
\end{align}
It is similarly found that 
\begin{align}
\notag\phi_{\textbf{p}}(x_i) &= Pr(i)\phi_{\textbf{p}}\left(x_M\right), \mbox{ for all $i$,}\\
\notag\mathbb{E}_{x_i}\left(   \phi_{\textbf{p}}(X_1)  \right) &= \phi_{\textbf{p}}(x_i), \mbox{ for all $i\neq M$.}
\end{align}
Putting everything together yields
\begin{align}\notag
\phi_{\textbf{p}}(x_i) 
= \frac{ i p M^2}{M- (1-p) (M-1)}, \mbox{ for all $i$,}
\end{align}
\begin{gather*}
\mathbb{E}_{x_i}\left(   \phi_{\textbf{p}}(X_1)  \right)
		= \begin{cases}
  \frac{ i p M^2}{M- (1-p) (M-1)}, & i \neq M\\
	\frac{ (M-1)p M^2}{M- (1-p) (M-1)}&  i = M.
\end{cases}
	\end{gather*}
Similar calculations yield
 \begin{align}
\notag\psi_{\textbf{p}}(x_i) 
& = \frac{ i p M}{M- (1-p) (M-1)}
, \mbox{ for all $i$,}
\end{align}
\begin{gather*}
\mathbb{E}_{x_i}\left(\psi_{\textbf{p}}(X_1)  \right)
		= \begin{cases}
  \frac{ i p M}{M- (1-p) (M-1)}, & i \neq M\\
	\frac{ (M-1)p M}{M- (1-p) (M-1)}&  i = M.
	\end{cases}
	\end{gather*}
Using the findings above it is easy to verify:  
(i) condition \eqref{I} is satisfied for all $i$ and all $p$ (note that this corresponds to the fact that the variance is always non-negative), 
(ii) condition \eqref{II} holds with equality for $i\neq M$,  
(iii) condition \eqref{III} holds with equality for $i=M$ if and only if $p=\frac{1}{M+1}$ (the inequality in condition \eqref{III} is of course trivially satisfied), and 
(iv) condition \eqref{II}  holds for $p=\frac{1}{M+1}$ and $i = M$. 
Hence, Corollary \ref{cor:NecSufCondConcave} implies that the strategy 
\begin{align}
{\hat {\textbf{p}}}= \left(1,0,...,0,\frac{1}{M+1}\right)^T, \label{p-var-eq-strat}
\end{align}
is an equilibrium. Simple calculations imply that \eqref{p-var-eq-strat} corresponds to
\begin{align}\notag
\phi_{\hat {\textbf{p}}}(x_i) = \frac{ i M}{2}, \mbox{ for all $i$,}
\end{align}
\begin{gather*}
\mathbb{E}_{x_i}\left(   \phi_{\hat {\textbf{p}}}(X_1)  \right)
		= \begin{cases}
	 \frac{ i M}{2}, & i \neq M\\
		\frac{   (M-1)M}{2}&  i = M,
	\end{cases}
	\end{gather*}
\begin{align}\notag
\psi_{\hat {\textbf{p}}}(x_i) = \frac{ i }{2}, \mbox{ for all $i$,}
\end{align}
\begin{gather*} 
\mathbb{E}_{x_i}\left(   \psi_{\hat {\textbf{p}}}(X_1)  \right)
		= \begin{cases}
	 \frac{ i }{2}, & i \neq M\\
		\frac{  M-1}{2}&  i = M.
	\end{cases}
	\end{gather*}	 
The corresponding equilibrium value function is
\begin{align}\notag
J_{\hat {\textbf{p}}}(x_i) = \frac{ i M}{2} - \left(\frac{ i}{2}\right)^2, \mbox{ for all $i$.}
\end{align}
The equilibrium value function in the case $M=100$ is depicted in Figure \ref{fig-var1}. 

{We remark that for mixed equilibria one should not in general expect equilibrium stability. This can be made rigorous in this example: If we consider ${\hat {\textbf{p}}}^{\epsilon}$ by just changing ${\hat {p}_M}$ to ${\hat {p}}^{\epsilon}_M=\frac{1}{M+1}+\epsilon$, the myopic adjustment process can be found explicitly using the calculations above. Indeed, for small enough $\epsilon$, it holds that}
	\[{\bar \Gamma}({\hat {\textbf{p}}}^{\epsilon})_x=\begin{cases}
	\hat {p}_x,&x\not=M,\\
	\hat {p}_M-\frac{(M-1)}{2}\epsilon,&x=M.
	\end{cases}\]
	Therefore, in case $M\geq 3$, iterating this, we see that the myopic adjustment process cannot converge to the fixed point ${\hat {\textbf{p}}}${
	and it does in fact not converge at all (at least when $M=3$).} 
 Hence, $\hat {\textbf{p}}$ is not locally stable.\begin{figure} 
\centering 
\includegraphics{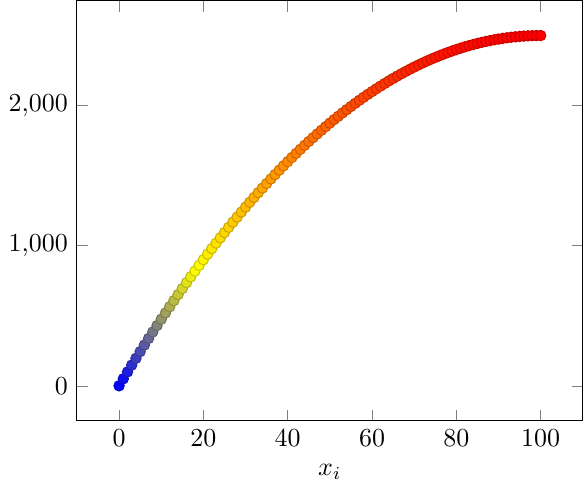}
								\caption{The equilibrium value function $x_i\mapsto J_{\hat {\textbf{p}}}(x_i)$ for $M=100$.}\label{fig-var1}
\end{figure}{
 The observation that (local) stability does not hold here is in line with other findings in the literature on games. Indeed, mixed equilibria are often found to have an unstable behavior. The next easy example, however, shows that this is not always the case:}

\begin{example} [A globally stable mixed equilibrium] \label{newExampleglob}{
 Consider the variance problem for the Markov chain defined in Figure \ref{MC-4}.
\begin{figure}[h]
\centering 
\includegraphics{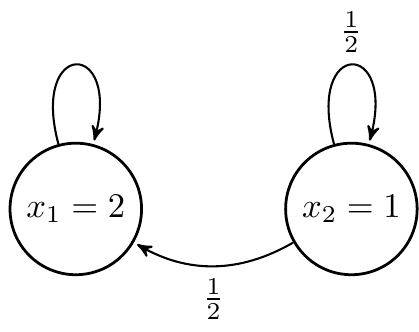}
\caption{The Markov chain $X$ in Example \ref{newExampleglob}.}\label{MC-4}
\end{figure}
 Similarly to the variance problem above it is easily verified that $\hat {\textbf{p}}=(1,1/3)$ is the only equilibrium and furthermore
	\[{\bar \Gamma}(p_1,p_2)=\left(1,\frac{1-p_2}{2}\right)\]
which is obviously a contraction with fixed point $(1,1/3)$, so that the equilibrium is globally stable.}
\end{example}

\section{Discussion and relation to the literature}\label{equilibrium-discussion}

The definitions of pure and mixed strategies as well as the equilibrium definition of the present paper are in line with the definitions of \cite{bayraktar2018} which studies mean-standard deviation and mean-variance stopping in a discrete time framework. 
A pure stopping strategy for continous time is in \cite{christensen2017finding,christensen2018time} defined as the entry time into a set in the state space. 
The stopping strategies considered in \cite{huang2018time,huang2019general,huang2019optimal,huang2017optimalDISC} are of the same type. 
Noticing that a pure stopping strategy ${\textbf{p}}$  corresponds to $$\tau_{\textbf{p}} = \min\{n\geq 0:X_n \in\{x \in E:{\textbf{p}}_x=1\} \},$$   
we see that our definition is in line with the literature. 
A mixed stopping strategy for continuous time is in \cite{christensen2018time} defined as the first stopping time of an $X$-associated Cox process, which is a natural continous time interpretation of the definition of a mixed stopping strategy in the present paper; see \cite[Section 2.1]{christensen2018time} for further arguments. 
The continuous time mixed equilibrium in \cite{christensen2018time} corresponds to a first order condition whose interpretation is 
in line 
with the present paper in the sense that a stopping strategy is an equilibrium if it is at no $x$ desirable to deviate from the equilibrium by using an alternative probability for stopping at $x$. The equilibrium definition in \cite{christensen2017finding} is analogous but in a framework considering only pure strategies. 
Further comparisons of definitions in the literature on time-inconsistent stopping is found in  e.g. \cite{christensen2017finding,christensen2018time}. 
In \cite{huang2017optimalDISC} a non-exponential discounting stopping problem in discrete time for stopping strategies that correspond to pure strategies in the sense of the present paper is studied and a method for finding equilibria similar to the myopic adjustment process, there called a fixed-point iteration, is used to establish equilibrium existence. 
The equilibrium definition of \cite{huang2017optimalDISC} differs from that of the present paper in the sense that it is not possible to deviate at $x$ from a proposed equilibrium stopping strategy if it 
suggests 
stopping at $x$, and hence the strategy of always stopping immediately is necessarily an equilibrium, see also the discussion in \cite[Section 2.1]{christensen2018time}. 
Similar frameworks and approaches to finding equilibria for time-inconsistent stopping problems in continuous time are studied in \cite{huang2018time,huang2019general,huang2019optimal}. We also note that a similar iteration approach is used to finding equilibria for a portfolio selection problem in \cite[Example 4.]{delong2018time}.

\appendix       
\section{Technical results}

In this section we derive properties for the function $\phi_\textbf{p}$ defined in \eqref{phi-psi-notation}. Analogous results are of course true for $\psi_\textbf{p}$ defined in \eqref{phi-psi-notation}.

\begin{lemma} \label{phi-p-expressed-with-Exp} For each $\textbf{p}$ and $x$ 
it 
holds that,
	\begin{align}
	\phi_\textbf{p}(x) & = p_xf(x) + (1-p_x)\mathbb{E}_x\left( \phi_{{ \textbf{p} }}(X_1) \right).\notag
	\end{align}
\end{lemma}
\begin{proof} 
This follows 
 from the definitions of $\textbf{p}$ and ${\tau_{\textbf{p}}}$.
\end{proof}

\begin{lemma} \label{lem1} For each $\textbf{p}$ and $x$, the following identities hold:
\begin{align}
\notag &	\phi_{\textbf{p}}(x) = \mathbb{E}_x\left(
	I_{\{\tau_{\textbf{p}} <\infty\}}\sum_{i\in\mathbb{N}_0} \prod_{j=1}^i (1-{p}_{X_{j-1}}){p}_{X_{i}}f(X_i)+
	I_{\{\tau_{\textbf{p}} =\infty\}} \lim_{n\rightarrow\infty}f(X_n)
	\right),\\
\notag &	\mathbb{E}_x\left( \phi_{\textbf{p}}(X_1)\right) = \mathbb{E}_x\left(
		I_{\{\tau_{\textbf{p}} <\infty\}}\sum_{i\in\mathbb{N}} \prod_{j=2}^i (1-{p}_{X_{j-1}}){p}_{X_{i}}f(X_i)+
	I_{\{\tau_{\textbf{p}} =\infty\}} \lim_{n\rightarrow\infty}f(X_n)
	\right),
	\end{align}
	where we use the convention $\prod_{j=k}^l:=1$ for $l<k$. Moreover, the functions
	\begin{align}
	{\textbf{p}} \in [0,1]^N \mapsto  \phi_{\textbf{p}}(x), \;{\textbf{p}} \in [0,1]^N \mapsto  \mathbb{E}_x\left( \phi_{\textbf{p}}(X_1)\right)\notag
	\end{align}
are, for each fixed $x\in E$, continuous.
\end{lemma}
\begin{proof} Using the notation \eqref{phi-psi-notation}, Fubini's Theorem and Assumption \ref{absorb-assum} we immediately obtain the identities. Recall that $\phi_\textbf{p}$ is independent of the choice of $p_x$ for each absorbing state $x\in E$. The continuity follows from majorized convergence.
\end{proof}

\begin{lemma} \label{lem-phi} Consider a function $\phi:E\rightarrow \mathbb{R}$ 
and a vector $\textbf{p}\in [0,1]^N$ 
with $p_x>0$ for each absorbing state $x\in E$ and suppose
	\begin{align}
	\phi(x) = {p}_xf(x)  + (1-{p}_x)\mathbb{E}_x\left(\phi(X_1)\right),  \mbox{ for each $x\in E$,}\label{lem1xx}
	\end{align}
	then 
	\begin{align}\notag
	\phi(x) = \phi_{\textbf{p}}(x), \mbox{ for each $x\in E$}.
	\end{align}

\end{lemma}
\begin{proof}
Note that $\tau_{\textbf{p}}<\infty$ a.s. Repeated substitution in \eqref{lem1xx} and the Markov property give, with a slight abuse of notation,
\begin{align}
\phi(x) 
\notag & = {p}_xf(x)  + (1-{p}_x)\mathbb{E}_x\left(\phi(X_1)\right)\\
\notag &=\mathbb{E}_x\left(\sum_{i\in\mathbb{N}_0} \prod_{j=1}^i (1-{p}_{X_{j-1}}){p}_{X_{i}}f(X_i)\right).
	\end{align}
	Now use Lemma \ref{lem1}. 
\end{proof}

\begin{lemma}\label{app-lemma1}  If $g:\mathbb{R}\rightarrow\mathbb{R}$	is a convex function then
	\begin{align} 
	q \mapsto (c_1q + c_2(1-q) + g(c_3q + c_4(1-q))\notag
	\end{align}
	is a convex function, for any constants $c_1,...,c_4$. The analogous result holds in the case $g$ is concave.

\end{lemma} 
\begin{proof} Follows directly from the definition convexity/concavity. 
\end{proof} 
\end{example}


\begin{thebibliography}{HD}






 
\bibitem{alia2019non}
I.~Alia.
\newblock \emph{A non-exponential discounting time-inconsistent stochastic optimal
  control problem for jump-diffusion}.
\newblock { Mathematical Control \& Related Fields}, 9(3):541--570, 2019.

\bibitem{balbus2018markov}
L.~Balbus, A.~Ja{\'s}kiewicz, and A.~S. Nowak.
\newblock \emph{Markov perfect equilibria in a dynamic decision model with
  quasi-hyperbolic discounting}.
\newblock { Annals of Operations Research}, pages 1--19, 2018.

\bibitem{bayraktar2019notions}
E.~Bayraktar, J.~Zhang, and Z.~Zhou.
\newblock \emph{On the notions of equilibria for time-inconsistent stopping problems
  in continuous time}.
\newblock {arXiv:1909.01112}, 2019.

\bibitem{bayraktar2018}
E.~Bayraktar, J.~Zhang, and Z.~Zhou.
\newblock \emph{Time consistent stopping for the mean-standard deviation
  problem---the discrete time case}.
\newblock { SIAM Journal on Financial Mathematics}, 10(3):667--697, 2019.

\bibitem{bensoussan2014time}
A.~Bensoussan, K.~Wong, S.~C.~P. Yam, and S.-P. Yung.
\newblock \emph{Time-consistent portfolio selection under short-selling prohibition:
  From discrete to continuous setting}.
\newblock { SIAM Journal on Financial Mathematics}, 5(1):153--190, 2014.

\bibitem{bielecki2005continuous}
T.~R. Bielecki, H.~Jin, S.~R. Pliska, and X.~Y. Zhou.
\newblock \emph{Continuous-time mean-variance portfolio selection with bankruptcy
  prohibition}.
\newblock { Mathematical Finance}, 15(2):213--244, 2005.

\bibitem{tomas-continFORTH}
T.~Bj{\"o}rk, M.~Khapko, and A.~Murgoci.
\newblock \emph{On time-inconsistent stochastic control in continuous time}.
\newblock { Finance and Stochastics}, 21(2):331--360, 2017.

\bibitem{tomas-disc}
T.~Bj{\"o}rk and A.~Murgoci.
\newblock \emph{A theory of {M}arkovian time-inconsistent stochastic control in
  discrete time}.
\newblock { Finance and Stochastics}, 18(3):545--592, 2014.

\bibitem{tomas-mean}
T.~Bj{\"o}rk, A.~Murgoci, and X.~Y. Zhou.
\newblock \emph{Mean-variance portfolio optimization with state-dependent risk
  aversion}.
\newblock { Mathematical Finance}, 24(1):1467--9965, 2014.

\bibitem{borgers1997learning}
T.~B{\"o}rgers and R.~Sarin.
\newblock \emph{Learning through reinforcement and replicator dynamics}.
\newblock { Journal of Economic Theory}, 77(1):1--14, 1997.

\bibitem{buonaguidi2015remark}
B.~Buonaguidi.
\newblock \emph{A remark on optimal variance stopping problems}.
\newblock { Journal of Applied Probability}, 52(4):1187--1194, 2015.

\bibitem{buonaguidi2018some}
B.~Buonaguidi and A.~Mira.
\newblock \emph{Some optimal variance stopping problems revisited with an application
  to the {Italian} {Ftse-Mib} stock index}.
\newblock { Sequential Analysis}, 37(1):90--101, 2018.

\bibitem{christensen2017finding}
S.~Christensen and K.~Lindensj{\"o}.
\newblock \emph{On finding equilibrium stopping times for time-inconsistent
  {M}arkovian problems}.
\newblock {SIAM Journal on Control and Optimization}, 56(6):4228--4255,
  2018.

\bibitem{christensen2019dividend}
S.~Christensen and K.~Lindensj{\"o}.
\newblock \emph{Moment constrained optimal dividends: precommitment \& consistent
  planning}.
\newblock {arXiv:1909.10749}, 2019.

\bibitem{christensen2018time}
S.~Christensen and K.~Lindensj{\"o}.
\newblock \emph{On time-inconsistent stopping problems and mixed strategy stopping
  times}.
\newblock { to appear in Stochastic Processes and their Applications, DOI:
  10.1016/j.spa.2019.08.010}, 2019+.

\bibitem{Czichowsky}
C.~Czichowsky.
\newblock \emph{Time-consistent mean-variance portfolio selection in discrete and
  continuous time}.
\newblock { Finance and Stochastics}, 17(2):227--271, 2013.

\bibitem{delong2018time}
L.~Delong.
\newblock \emph{Time-inconsistent stochastic optimal control problems in insurance
  and finance}.
\newblock { Collegium of Economic Analysis Annals}, (51):229--254, 2018.

\bibitem{delong2015instantaneous}
L.~Delong and A.~Pelsser.
\newblock \emph{Instantaneous mean-variance hedging and sharpe ratio pricing in a
  regime-switching financial model}.
\newblock { Stochastic Models}, 31(1):67--97, 2015.

\bibitem{detemple1992optimal}
J.~B. Detemple and F.~Zapatero.
\newblock \emph{Optimal consumption-portfolio policies with habit formation}.
\newblock { Mathematical Finance}, 2(4):251--274, 1992.

\bibitem{Duraj2017Optimal}
J.~Duraj.
\newblock \emph{Optimal stopping with general risk preferences}.
\newblock { SSRN preprint:2897765}, 2017.

\bibitem{englezos2009utility}
N.~Englezos and I.~Karatzas.
\newblock \emph{Utility maximization with habit formation: Dynamic programming and
  stochastic {PDE}s}.
\newblock { Siam Journal on Control and Optimization}, 48(2):481--520, 2009.

\bibitem{erev1998predicting}
I.~Erev and A.~E. Roth.
\newblock \emph{Predicting how people play games: Reinforcement learning in
  experimental games with unique, mixed strategy equilibria}.
\newblock { American Economic Review}, pages 848--881, 1998.

\bibitem{fan1952fixed}
K.~Fan.
\newblock \emph{Fixed-point and minimax theorems in locally convex topological linear
  spaces}.
\newblock { Proceedings of the National Academy of Sciences},
  38(2):121--126, 1952.

\bibitem{gad2019optimal}
K.~S.~T. Gad and P.~Matom{\"a}ki.
\newblock \emph{Optimal variance stopping with linear diffusions}.
\newblock { to appear in Stochastic Processes and their Applications, DOI:
  10.1016/j.spa.2019.07.001}, 2019+.

\bibitem{gad2015variance}
K.~S.~T. Gad and J.~L. Pedersen.
\newblock \emph{Variance optimal stopping for geometric {L{\'e}vy} processes}.
\newblock { Advances in Applied Probability}, 47(1):128--145, 2015.

\bibitem{he2013optimal}
L.~He and Z.~Liang.
\newblock \emph{Optimal investment strategy for the {DC} plan with the return of
  premiums clauses in a mean--variance framework}.
\newblock { Insurance: Mathematics and Economics}, 53(3):643--649, 2013.

\bibitem{he2019optimal}
X.~D. He, S.~Hu, J.~Ob{\l}{\'o}j, and X.~Y. Zhou.
\newblock \emph{Optimal exit time from casino gambling: Strategies of precommitted
  and naive gamblers}.
\newblock { SIAM Journal on Control and Optimization}, 57(3):1845--1868,
  2019.

\bibitem{huang2018time}
Y.-J. Huang and A.~Nguyen-Huu.
\newblock \emph{Time-consistent stopping under decreasing impatience}.
\newblock { Finance and Stochastics}, 22(1):69--95, 2018.

\bibitem{huang2019general}
Y.-J. Huang, A.~Nguyen-Huu, and X.~Y. Zhou.
\newblock \emph{General stopping behaviors of naive and noncommitted sophisticated
  agents, with application to probability distortion} (forthcoming:
  Doi:10.1111/mafi.12224).
\newblock { Mathematical Finance}.

\bibitem{huang2019optimal}
Y.-J. Huang and X.~Yu.
\newblock \emph{Optimal stopping under model ambiguity: a time-consistent equilibrium
  approach}.
\newblock {arXiv:1906.01232}, 2019.

\bibitem{huang2018strong}
Y.-J. Huang and Z.~Zhou.
\newblock \emph{Strong and weak equilibria for time-inconsistent stochastic control
  in continuous time}.
\newblock {arXiv:1809.09243}, 2018.

\bibitem{huang2017optimal}
Y.-J. Huang and Z.~Zhou.
\newblock \emph{Optimal equilibria for time-inconsistent stopping problems in
  continuous time}.
\newblock {Mathematical Finance. 2019; 1– 32. https://doi.org/10.1111/mafi.12229}




\bibitem{huang2017optimalDISC}
Y.-J. Huang and Z.~Zhou.
\newblock \emph{The optimal equilibrium for time-inconsistent stopping problems---the
  discrete-time case}.
\newblock {SIAM Journal on Control and Optimization}, 57(1):590--609, 2019.

\bibitem{kosfeld2002myopic}
M.~Kosfeld, E.~Droste, and M.~Voorneveld.
\newblock \emph{A myopic adjustment process leading to best-reply matching}.
\newblock { Games and Economic Behavior}, 40(2):270--298, 2002.

\bibitem{kronborg2015inconsistent}
M.~T. Kronborg and M.~Steffensen.
\newblock \emph{Inconsistent investment and consumption problems}.
\newblock { Applied Mathematics \& Optimization}, 71(3):473--515, 2015.

\bibitem{landriault2018equilibrium}
D.~Landriault, B.~Li, D.~Li, and V.~R. Young.
\newblock \emph{Equilibrium strategies for the mean-variance investment problem over
  a random horizon}.
\newblock { SIAM Journal on Financial Mathematics}, 9(3):1046--1073, 2018.

\bibitem{li2013optimal}
Y.~Li and Z.~Li.
\newblock \emph{Optimal time-consistent investment and reinsurance strategies for
  mean--variance insurers with state dependent risk aversion}.
\newblock { Insurance: Mathematics and Economics}, 53(1):86--97, 2013.

\bibitem{lindensjo2017timeinconHJB}
K.~Lindensj{\"o}.
\newblock \emph{A regular equilibrium solves the extended {HJB} system}.
\newblock { Operations Research letters}, 47(5):427--432, 2019.



\bibitem{nash1950equilibrium}
J.~Nash.
\newblock \emph{Equilibrium points in n-person games}.
\newblock { Proceedings of the national academy of sciences}, 36(1):48--49, 1950.




\bibitem{nutz2019onditional}
M.~Nutz and Y.~Zhang.
\newblock \emph{Conditional optimal stopping: A time-inconsistent optimization}.
\newblock (to appear in Annals of Applied Probability) {arXiv:1901.05802}, 2019.

\bibitem{pedersen2011explicit}
J.~L. Pedersen.
\newblock \emph{Explicit solutions to some optimal variance stopping problems}.
\newblock { Stochastics An International Journal of Probability and
  Stochastic Processes}, 83(4-6):505--518, 2011.

\bibitem{pedersen2016optimal}
J.~L. Pedersen and G.~Peskir.
\newblock \emph{Optimal mean--variance selling strategies}.
\newblock { Mathematics and Financial Economics}, 10(2):203--220, 2016.

\bibitem{pedersen2013optimal}
J.~L. Pedersen and G.~Peskir.
\newblock \emph{Optimal mean-variance portfolio selection}.
\newblock { Mathematics and Financial Economics}, 11(2):137--160, 2017.

\bibitem{schoneborn2015optimal}
T.~Sch{\"o}neborn.
\newblock \emph{Optimal trade execution for time-inconsistent mean-variance criteria
  and risk functions}.
\newblock { SIAM Journal on Financial Mathematics}, 6(1):1044--1067, 2015.

\bibitem{selten1965spieltheoretische}
R.~Selten.
\newblock \emph{Spieltheoretische behandlung eines oligopolmodells mit
  nachfragetr{\"a}gheit: Teil i: Bestimmung des dynamischen
  preisgleichgewichts}.
\newblock { Zeitschrift f{\"u}r die gesamte Staatswissenschaft/Journal of
  Institutional and Theoretical Economics}, (H. 2):301--324, 1965.

\bibitem{selten1975reexamination}
R.~Selten.
\newblock \emph{Reexamination of the perfectness concept for equilibrium points in
  extensive games}.
\newblock { International journal of game theory}, 4(1):25--55, 1975.

\bibitem{MR2374974}
A.~N. Shiryaev.
\newblock { \emph{Optimal stopping rules}}, volume~8 of { Stochastic Modelling
  and Applied Probability}.
\newblock Springer-Verlag, Berlin, 2008.
\newblock Translated from the 1976 Russian second edition by A. B. Aries,
  Reprint of the 1978 translation.

\bibitem{strotz}
R.~Strotz.
\newblock \emph{Myopia and inconsistency in dynamic utility maximization}.
\newblock { The Review of Economic Studies}, 23(3):165--180, 1955.

\bibitem{tan2018failure}
K.~S. Tan, W.~Wei, and X.~Y. Zhou.
\newblock \emph{Failure of smooth pasting principle and nonexistence of equilibrium
  stopping rules under time-inconsistency}.
\newblock {arXiv:1807.01785}, 2018.

\bibitem{van2018time}
P.~M. Van~Staden, D.-M. Dang, and P.~A. Forsyth.
\newblock \emph{Time-consistent mean--variance portfolio optimization: A numerical
  impulse control approach}.
\newblock { Insurance: Mathematics and Economics}, 83:9--28, 2018.

\bibitem{vigna2014efficiency}
E.~Vigna.
\newblock \emph{On efficiency of mean--variance based portfolio selection in defined
  contribution pension schemes}.
\newblock { Quantitative finance}, 14(2):237--258, 2014.

\bibitem{xu2013optimal}
Z.~Q. Xu, X.~Y. Zhou, et~al.
\newblock \emph{Optimal stopping under probability distortion}.
\newblock { The Annals of Applied Probability}, 23(1):251--282, 2013.

\bibitem{yan2019open}
T.~Yan and H.~Y. Wong.
\newblock \emph{Open-loop equilibrium strategy for mean--variance portfolio problem
  under stochastic volatility}.
\newblock { Automatica}, 107:211--223, 2019.

\bibitem{yan2019time}
W.~Yan and J.~Yong.
\newblock \emph{Time-inconsistent optimal control problems and related issues}.
\newblock In { Modeling, Stochastic Control, Optimization, and
  Applications}, pages 533--569. Springer, 2019.

\bibitem{yu2017optimal}
X.~Yu.
\newblock \emph{Optimal consumption under habit formation in markets with transaction
  costs and random endowments}.
\newblock { The Annals of Applied Probability}, 27(2):960--1002, 2017.

\bibitem{zeng2016robust}
Y.~Zeng, D.~Li, and A.~Gu.
\newblock \emph{Robust equilibrium reinsurance-investment strategy for a
  mean--variance insurer in a model with jumps}.
\newblock { Insurance: Mathematics and Economics}, 66:138--152, 2016.









\end{thebibliography}
\end{document}